\documentclass{article}

\usepackage{graphicx}
\usepackage{multirow}%
\usepackage{amsmath,amssymb,amsfonts,amsthm}
\usepackage{hyperref,cleveref}

\usepackage{xcolor}%
\hypersetup{
	colorlinks,
	linkcolor={red!50!black},
	citecolor={blue!50!black},
	urlcolor={blue!80!black}
}
\theoremstyle{plain}
\newtheorem{theorem}{Theorem}[section]
\newtheorem{lemma}[theorem]{Lemma}

\newtheorem{proposition}[theorem]{Proposition}
\newtheorem{definition}[theorem]{Definition}
\newtheorem{assumption}[theorem]{Assumption}
\newtheorem{remark}[theorem]{Remark}

\DeclareMathOperator*{\du}{d\!}

\newcommand{\dive}{\operatorname{div}}
\newcommand{\meas}{\operatorname{meas}}

\title{Stability analysis of the  Navier--Stokes velocity tracking problem with bang-bang controls\thanks{A.D.C. is supported by the Alexander von Humboldt
		Foundation with an Alexander von Humboldt research fellowship and the grant FWF I4571-N.
		N.J was supported by the FWF grants P-31400-N32 and I4571-N.
		\v S. N. and J.S.H.S.  have been supported by  Praemium Academiæ of \v S. Ne\v casov\' a. \v S. N. was also supported by the Czech Science Foundation (GA\v CR) through projects GC22-08633J. The Institute of Mathematics, CAS is supported by RVO:67985840.}}
\author{Alberto Dom{\'i}nguez Corella \and Nicolai Jork \and \v S{\'a}rka Ne\v casov\'a \and John Sebastian H. Simon}
\date{}

\begin{document}

\maketitle

\begin{abstract}
  	This paper focuses on the stability of solutions for a velocity-tracking problem associated with the two-dimensional Navier--Stokes equations. The considered optimal control problem does not possess any regularizer in the cost, and hence bang-bang solutions can be expected. We investigate perturbations that account for uncertainty in the tracking data and the initial condition of the state, and  analyze the convergence rate of solutions when the original problem is regularized by the Tikhonov term. The stability analysis relies on the H\"{o}lder subregularity of the optimality mapping, which stems from the necessary conditions of the problem.
\end{abstract}

\section{Introduction}

Let $\Omega\subset\mathbb R^2$ be a bounded domain and $T>0$ and let  $ u_a, u_b:\Omega\times(0,T)\to\mathbb R^2$ be bounded functions.  We consider the  \textit{control set}
 \begin{align}\label{conset}
		\mathcal U:=\left\lbrace  u\in L^\infty(Q)^2:\,\,  u_a(x,t)\le u(x,t)\le  u_b(x,t)\,\,\,\,\text{for a.e. }(x,t)\in Q:=\Omega\times(0,T) \right\rbrace.
\end{align} 
For each control $ u\in\mathcal U$,  interpreted as a force per unit mass acting on the fluid, there is an associated state $ {y_{u}}:Q\to\mathbb R^2$, which is the fluid velocity, satisfying the Navier--Stokes equation
 \begin{align}\label{dynsys} {
			\left\{ \begin{aligned}
				&\partial_t{y_u}- \nu\Delta {y_u}+ ({y_u}\cdot\nabla){y_u} + \nabla {p_u} = u\text{ in }Q,
				\\ &\dive {y_u}=0\,\,\,\,\text{in Q},\,\,  {y_u}=0\,\,\,\,\text{on $\Sigma$},\,\,\, {y_u}(\cdot,0)= y_0\,\,\,\text{in $\Omega$}.
			\end{aligned}
			\right.
		}
\end{align} 
Here, $ y_0$ denotes the initial velocity field, $\nu>0$ is the so-called kinematic viscosity parameter, and $\Sigma:=\partial\Omega\times(0,T)$  the lateral boundary of the open cylinder $Q$.
We consider the classical  \textit{velocity tracking problem} 
 \begin{align}\label{optpro}
		\min_{u\in\mathcal U}\frac{1}{2}\int_{0}^T\int_\Omega |{y_u}(x,t) -  y_{d}(x,t)|^2 \, dx\, dt.
\end{align} 

 {Optimal control problems often involve a regularizing term that aids in the formulation and analysis of the control strategies and the numerical implementation. However, there exist certain scenarios where this regularizing term is absent, such is the case for affine problems, which leads to a different class of solutions that pose additional challenges. In particular, the optimal control for problem (\ref{conset})-(\ref{optpro}) may generically exhibit a bang-bang behavior, i.e. the control $u = (u^1,u^2)$ may satisfy $u^i(x,t) \in \{ u^i_a(x,t), u^i_b(x,t) \}$ for a.e. $(x,t)\in Q$. 
	These kinds of optimal control problems have attracted attention due to their implications for the solution of the underlying dynamical system (\ref{dynsys}). While a regularization on the objective functional could offer certain mathematical advantages, the problem at hand steers the state solution towards the desired velocity better as highlighted previously in \cite{Vtyt1} and \cite{CKbangbang}.}

This paper focuses on the stability of solutions for problem (\ref{conset})-(\ref{optpro}). To highlight our contributions, we recall previous literature related to the problem. We begin with some previous works on optimal control of the Navier--Stokes equations where the regularization term appears. As second-order conditions are intimately related to stability, we mention the work \cite{SecWT} on sufficient conditions. In regard to stability, we mention the works \cite{CK1,CK2}, where error estimates for a Galerkin time-stepping scheme were established.  {We want to mention also the recent paper \cite{FO2023}, where a regularized point-wise-in-time tracking type objective functional is considered.}

In \cite{Vtyt1}, the authors presented a systematic approach to the mathematical analysis and numerical
approximation of tracking the velocity for Navier--Stokes flows with bounded distributed controls without the classic Tikhonov regularization term --  {which is precisely written as the addition of the square of the $L^2$-norm of the control to the objective functional, see \eqref{optproper}}. This analysis includes first-order necessary conditions as well as discretization schemes.  {In \cite{CK2019} the authors consider control problems subject to the stationary Navier--Stokes equation with a sparsity promoting term appearing in the objective functional. The key novelty, besides considering the Navier--Stokes equation, in \cite{CK2019} is that the controls are measure valued. The authors establish first-order necessary and second-order sufficient conditions for local optimality. Subsequently, they study the stability of the optimal states for perturbations appearing distributed in the Navier--Stokes equation and the observational data, see \cite[Theorem 5.1, Theorem 5.3]{CK2019}. In \cite{CK2021} the authors of \cite{CK2019} continue their study of measure-valued control problems, this time for the evolutionary Navier--Stokes equation where the main result is the establishment of first-order necessary and second-order sufficient optimality conditions. This time, there is not a sparsity-promoting term included and the problem is closer to the problem considered in this paper.} Regarding the second-order sufficient conditions and stability of this type of problem, we mention the work \cite{CKbangbang}, where a fully discrete scheme based on discontinuous
(in time) Galerkin approach combined with conformal finite element subspaces in space, is proposed and analyzed. Besides this work, to the best of our knowledge, there are not any other works in stability concerned with bang-bang controls of the velocity tracking problem (\ref{conset})-(\ref{optpro}).

The study of bang-bang controls for problems constrained by partial differential equations has not been thoroughly explored. It was only in 2012 that a second-order analysis for a semilinear elliptic equation was presented \cite{casas2012}. The investigation into the stability analysis of bang-bang minimizers for problems constrained by partial differential equations originated with \cite{Hinzebb}, focusing on the error estimates for discretizing the optimization problem. Since then, several papers have addressed other types of problems involving bang-bang minimizers.
For parabolic problems, the work of N. von Daniels and M. Hinze \cite{Par} explores the accuracy of a variational discretization, while E. Casas and F. Tr{\"o}ltzsch \cite{ParCasas} examined stability with respect to initial data.
In the context of optimal control problems governed by ordinary differential equations, the stability analysis of bang-bang minimizers was explored in \cite{VQreg}. This study investigated the stability of the first-order necessary conditions using the metric regularity property, considering assumptions that involved $L^1$-growths. These assumptions are similar to the classic coercivity condition,  {also known as the Legenrdre-Clebsch condition}, but with certain modifications.

This article aims to investigate the effect of perturbations on problem (\ref{conset})-(\ref{optpro}). Specifically, we focus on perturbations that arise from uncertainties in tracking data and initial state conditions. Additionally, we consider the classic Tikhonov regularization term as a perturbation of the original problem. We provide H\"older  estimates for the rates of convergence. To elaborate on our main result, suppose that we do not know the exact initial condition nor the datum we are tracking, however, we have an approximation of these data, i.e., $w_{0}$ and $w_d$. We can then solve the $\varepsilon$-regularized problem
 \begin{align}\label{optproper}
		\min_{u\in\mathcal U}\frac{1}{2}\int_{0}^T\int_\Omega |{y_u}(x,t) -  w_{d}(x,t)|^2 \, dx\, dt +  \frac{\varepsilon}{2}\int_{0}^T\int_\Omega |u(x,t)|^2\, dx\,dt
\end{align} 
subject to (\ref{dynsys}) with $w_0$ as initial datum and obtain a solution $\widehat u$.  Our main result (Theorem \ref{Mainthm}) gives the estimate 
 \begin{align*}
		\|\widehat u-\bar u\|_{L^1(Q)^2}\le \kappa\Big(\|y_0-w_0\|_{W^{2-\frac{2}{\bar s},s}_{0.\sigma}(\Omega)^2} +\|y_d-w_d\|_{L^2(Q)}+ \varepsilon\Big)^{\frac{1}{\mu}},
\end{align*} 
for a reference solution $\bar u$ of problem (\ref{conset})-(\ref{optpro}), under a growth assumption, in particular, Assumption \ref{growthasu} and the growth assumption in \Cref{theorem:suffsec}. We refer the reader to Theorem \ref{Mainthm} for more precise details, and the technicalities surrounding the  statement.

In our analysis, we employ the concept of strong Hölder subregularity of a set-valued mapping associated with the optimality of the optimal control problem \cite[Section 3I]{DontRock}. The stability of the first-order necessary conditions is investigated as a property of a set-valued mapping that encapsulates the generalized equation satisfied by local minimizers. This property is also referred to as strong (metric) $\theta$-subregularity in the literature, see \cite[Section 4]{Dontsubreg}.  {In previous works, this property has been shown to imply stability for solutions of} optimal control problems, particularly those constrained by partial differential equations. Subregularity results for semilinear problems can be found in \cite{Alellip,corella2023}. However, to the best of our knowledge, such results have not been explored for systems involving the Navier--Stokes equations, which constitutes one of the key contributions of this paper. It is important to note that the low regularity solutions of the Navier--Stokes equations present additional challenges in the analysis, which demand the development of distinct estimates compared to papers such as \cite{Alellip,corella2023,VQreg}.

In addition, we dedicate a self-contained appendix to abstract results in metric subregularity; we formalize the tools used to establish stability in the presence of perturbations. Importantly, these tools are designed to be applicable to a wide range of optimal control problems, extending beyond those constrained by the Navier--Stokes equations.
Previous papers that have addressed the subregularity of the optimality mapping have mainly focused on providing sufficient conditions. However, in this paper, we present an abstract result for necessary conditions, see Theorem \ref{Neccondcri}. This is a  novel contribution and  is significant because it was previously unknown whether subregularity could imply a growth condition. Furthermore, our result clearly establishes that this type of stability implies the bang-bang nature of optimal controls. Due to the general nature of this result, it can be applied to several previously studied optimal control problems, including those examined in \cite{Alellip,corella2023,VQreg}.


This article is structured as follows: Section \ref{sec:2} discusses the functional theoretic tools that are essential for the analysis of the state equations. Section \ref{sec:3} focuses on the analysis of the Navier--Stokes equations, encompassing the existence and regularity of solutions, as well as the examination of the linear system known as the Oseen equations and its dual system. In Section \ref{sec:4}, we delve into the analysis of the optimal control problem, providing first-order necessary conditions and second-order sufficient conditions. Finally, in Section \ref{sec:5}, we establish the stability of optimal controls in the presence of perturbations, which include variations in the desired velocity, the initial state of the governing equations, and a Tikhonov perturbation in the objective functional. The abstract results on subregularity are presented in the Appendix.


\section{Preliminaries}\label{sec:2}

{For a given real normed space $X$,   its topological dual is denoted by $X^*$ and  $\langle \cdot, \cdot \rangle_{X^*,X}:X^*\times X\to\mathbb R$ denotes their duality pairing.}
The domain $\Omega$ under consideration is a connected bounded subset of $\mathbb R^2$ with boundary $\partial\Omega$ of class $C^3$. The unit normal vector field is denoted by $ n:\partial\Omega\to\mathbb R^2$.

\vspace{.1 in}\noindent{\bf Sobolev spaces.}
{For a measurable set $E$, we consider the usual Lebesgue spaces $L^s(E)$ of $s$-integrable functions for $s\in[1,\infty)$, and $L^\infty(E)$ the space of essentially bounded functions. These spaces are endowed with their standard norms, denoted by $\|\cdot\|_{L^s(E)^d}$ for $s\in [1,+\infty]$ and $d = 1,2,2\times 2$. We shall use the notation $(\cdot,\cdot)_E$ for the $L^2(E)^d$ inner product.}

For $m\in\mathbb N$ and $1\le s\le \infty$,  $ W^{m,s}(\Omega)$ denotes the space of all functions in $L^p(\Omega)$ with all of its weak derivatives of order $m$ belonging $L^p(\Omega)$. We write  $ W^{0,p}(\Omega):= L^p(\Omega)$ and  $ H^m(\Omega):=  W^{m,2}(\Omega)$. The norms in these spaces will be denoted as $\|\cdot\|_{W^{m,s}(\Omega)^d}$.

The zero trace Sobolev spaces are denoted by $ W_{0}^{m,s}(\Omega)$. Poincar{\'e} inequality {implies} that the seminorm in $W^{1,s}(\Omega)$, which is  {defined for any $\varphi \in W^{1,s}(\Omega)$ as $|\varphi|_{W^{1,s}(\Omega)} = \|\nabla \varphi\|_{L^s(\Omega)}$}, is a norm in $W_{0}^{1,s}(\Omega)$ and is equivalent to the usual norm. Hence, from hereon, when we refer to the norm in $W_{0}^{1,s}(\Omega)$ we mean it to be the seminorm. As usual, for the Hilbertian case, we write $H^{1}_0(\Omega):= W_{0}^{1,2}(\Omega)$ and $ H^{-1}(\Omega):= H_0^1(\Omega)^* $.

Let us recall some embeddings that are vital for the upcoming analyses.
Rellich--Kondrachov embedding theorem gives us the compact embedding $W^{m,s_1}(\Omega) \hookrightarrow L^{s_2}(\Omega)$ if either of the two cases hold: i. $m>0$ and $1 \le s_2 < 2s_1/(2-ms_2)$; and ii. $2=ms_1$ and  {$1\le s_2 < + \infty$}. Similarly, we get the {continuous} embedding $W^{j+m,s}(\Omega) \hookrightarrow C^{j}(\overline{\Omega})$ if either $ms \ge 2$, it becomes compact when $ms > 2$. For references, see \cite[Theorem 6.3, Part I]{adams2003} for the former, and \cite[Theorem 6.3, Part III]{adams2003} for the latter.

\vspace{.1 in}\noindent{\bf Solenoidal spaces.} 
To take into account the incompressibility condition, we consider the following solenoidal spaces
 \begin{align*}
		& W^{m,s}_{0,\sigma} := \left\{ \psi\in  W^{m,s}_0(\Omega)^2:\,   \dive\psi = 0 \text{ in }\Omega   \right\},\\
		&H_s := \left\{\psi\in L^s(\Omega)^2: \dive\psi = 0 \text{ in }L^s(\Omega),\, \psi\cdot n = 0 \text{ on }\partial\Omega \right\}.
\end{align*} 
We write $V_s:= W^{1,s}_{0,\sigma}$ and in the Hilbertian case, we use the notations $ V:=  V_2$ and $ H:= H_2$.

The spaces $ V$ and $ H$ form a Gelfand triple $( V, H, V^*)$, i.e., the embeddings
 \begin{align*}
		{V\hookrightarrow  H\cong  H^*\hookrightarrow  V^*} 
\end{align*} 
are dense and continuous. Moreover, the first embedding is compact due to Rellich--Kondrachov Theorem; and by Schauder Theorem, the second one is also compact. 

\vspace{.1 in}\noindent{\bf Auxiliary spatial operators.} The orthogonal complement of $ H$ can be characterized as $
H^\bot=\{\psi\in   L^2(\Omega)^2:\,\psi = \nabla  q \,\,\,\,\text{for some $q\in H^1(\Omega)^2$} \}.$
The representation  $ L^2(\Omega)^2=  H\oplus  H^\bot$ is called \textit{the Helmholtz--Leray decomposition}.
We define the \textit{Leray projection operator} $P:  L^2(\Omega)^2 \to  H$ by $P\psi = \psi_1$, where $\psi_1$ is the unique element of $ H$ such that $\psi-\psi_1$ belongs to $ H^\bot$.

{The Stokes operator $A:V\cap H^2(\Omega)^2\subset H\to H$ is thus defined as $A \psi = -P\Delta \psi$. 
	We can also look at the Stokes operator as linear operator from $V\cap H^2(\Omega)^2$ to $V^*$, i.e., $\langle A\psi_1,\psi_2 \rangle_V = (\nabla\psi_1,\nabla\psi_2)_\Omega$ for any $\psi_1\in V\cap H^2(\Omega)^2$ and $\psi_2\in V$.}

To deal with the nonlinearity caused by the convective {term}, we introduce the trilinear form $b: H^1(\Omega)^2\times  H^1(\Omega)^2\times  H^1(\Omega)^2\to \mathbb{R}$ defined as $b( \psi_1,  \psi_2,  \psi_3) = ((\psi_1\cdot\nabla)\psi_2, \psi_3)_\Omega$. Using H{\"o}lder inequality and Rellich--Kondrachov embedding one can see that $b$ is continuous. Furthermore, for any $\psi,\psi_1,\psi_2\in V$ we have $b(\psi, \psi_1,\psi_2) = -b(\psi, \psi_2,\psi_1)$. As a consequence of H{\"o}lder and Gagliardo--Nirenberg inequalities we also get
 \begin{align}
		b(\psi_1, \psi_2,\psi_3) \le c \|\psi_1\|_H^{1/2}\|\psi_1\|_V^{1/2}\|\psi_2\|_V\|\psi_3\|_H^{1/2}\|\psi_3\|_V^{1/2}{\quad \forall \psi_1,\psi_2,\psi_3 \in V}.\label{ladyzhen}
\end{align} 
 {For more details we refer to \cite{girault1986} and \cite{temam1984}.}

\vspace{.1 in}\noindent{\bf Bochner spaces.}
Let $T>0$ and  $X$ be a Banach space. We use the notation $C(\overline{I};X)$ for functions from $I$ to $X$ that can be continuously extended to $[0,T]$. For $s\in[1,\infty]$, we denote by $L^s(I; X)$ the usual space of functions $\psi:I\to X$ such that $t\to \|\psi(t)\|_{X}$ belongs to $L^s(0,T)$. The norm in $L^s(I; X)$ will be denoted as $\|\cdot\|_{L^s(X)}$, and the duality pairing of $L^s(I;X)$ and its dual as $\langle\cdot,\cdot\rangle_{L^s(X)}$. For $s\in[1,\infty]$ and $m\in\mathbb N$, the space $W^{m,s}(I;X)$ consists of functions $\psi\in L^s(I;X)$ whose distributional time derivative  $\partial_t^m \psi$ belongs to $L^s(I;X)$. For $m\in\mathbb N$, we write $H^m(I;X) := W^{m,2}(I;X)$. Note that the spaces $L^s(I;L^s(\Omega)^2)$ for $1\le s <\infty$ can be identified as the space $L^s(Q)^2$.

 {We also introduce the spaces $W^\alpha(I):= L^2(I;V)\cap W^{1,\alpha}(I;V^*)$ that take into account the distributional time derivatives  with the norm $\|\cdot\|_{W^\alpha(I)} := \|\cdot \|_{L^2(I;V)} + \|\cdot\|_{W^{1,\alpha}(I;V^*)}$. For simplicity, we write $W(I)$ when $\alpha = 2$.}

To aid us in our analyses, we recall some embedding theorems from \cite[Theorem 3]{amann2001}. We start with two real Banach spaces $X_0$ and $X_1$ such that $X_1\hookrightarrow X_0$ is dense. We denote by $X_{\theta,s}:= (X_0,X_1)_{\theta,q}$ their real interpolation with exponential functor $0<\theta<1$, where $1\le q \le +\infty$. Amman's theorem gives us an embedding for the space $W^1_s(I,(X_0,X_1)):= L^s(I;X_1)\cap W^{1,s}(I;X_0)$. In fact, if $r\in\mathbb{R}$ with $1/s<r<1$ and $0\le \theta < 1-r $ then $W^1_s(I,(X_0,X_1)) \hookrightarrow C^{0,r-1/s}(I;X_{\theta,1})$. Furthermore, if $X_1\hookrightarrow X_0$ is compact then $W^1_s(I,(X_0,X_1)) \hookrightarrow C^{0,r-1/s}(\overline{I};X_{\theta,1})$ is also compact. By virtue of trace theorem one also gets $W^1_s(I,(X_0,X_1)) \hookrightarrow C(\overline{I};X_{1-1/s,1})$.

Some of the direct consequences of the embedding above are the following:
\begin{itemize}
	\item[$\bullet$] when $X_0 = V^*$ and $X_1 = V$ so that taking $s = 2$ on the second embedding above we have $X_{1/2,1} \hookrightarrow X_{1/2,2} = (V^*,V)_{1/2,2} = H$ which implies that the embedding $W(I) = W^1_2(I,(V^*,V)) \hookrightarrow C(\overline{I};H)$ is compact;
	\item[$\bullet$] when $X_0 = L^s(\Omega)^2$ and $X_1 = W^{2,s}(\Omega)^2$, we have, from \cite[Theorem 7.31]{adams2003}, that $X_{1-1/s,1} \hookrightarrow W^{2-2/s,s}(\Omega)^2$ and hence the continuous embeddings
	 \begin{align*}
			W^{2,1}_s:= L^s(I;W^{2,s}(\Omega)\cap V_s)\cap W^{1,s}(I;L^s(\Omega)^2)  \hookrightarrow   C(\overline{I};W^{2-2/s,s}_{0,\sigma}(\Omega)^2).
	\end{align*} 
\end{itemize}  
Because of the divergence-free and zero trace assumptions on the elements of $W^{2,1}_s$, we used $W^{2-2/s,s}_{0,\sigma}(\Omega)^2$ instead of $W^{2-2/s,s}(\Omega)^2$. Furthermore, by virtue of Rellich--Kondrachov embedding theorem, we see that the embedding $W^{2-2/s,s}(\Omega)^2 \hookrightarrow C(\overline{\Omega})^2$ is compact, whenever $(2-2/s)s > 2$. Hence, the embedding $W^{2,1}_s \hookrightarrow C(\overline{Q})^2$ is also compact, whenever $s>2$. Similarly, $W^{2-2/s,s}(\Omega)^2 \hookrightarrow C^1(\overline{\Omega})^2$ is compact if $(1-2/s)s > 2$. Therefore, we have the compact embedding $W^{2,1}_s \hookrightarrow C(\overline{I};C^1(\overline{\Omega})^2)$ whenever $s > 4$.

We end this section by introducing some operators which help in the forthcoming analyses. We begin with the time-dependent extension of the Stokes operator, i.e., $A:L^2(I;V)\to L^2(I,V^*)$ defined as
 \begin{align*}
		\langle Au, v\rangle_{L^2(V^*),L^2(V)} & = \int_0^T \langle (Au)(t),v(t) \rangle_V \du t  = \int_0^T (\nabla u(t),\nabla v(t))_\Omega \du t.
\end{align*} 
In connection with the trilinear form we mentioned previously, we introduce the bilinear operator $B:W(I)\times W(I) \to L^2(I;V^*)$ defined as 
 \begin{align*}
		\langle B(u, v),w\rangle_{L^2(V^*),L^2(V)} = \int_0^T b(u(t),v(t),w(t)) \du t.
\end{align*} 
Indeed, the membership of $B(u,v)$ to $L^2(I;V^*)$ for $u,v\in W(I)$ follows from H{\"o}lder inequality and \eqref{ladyzhen}
 \begin{align*}
		|\langle B(u, v),w\rangle_{L^2(V^*),L^2(V)}| 
		&\le c\|u\|_{W(I)}\|v\|_{W(I)}\|w\|_{L^2(V)}.
\end{align*} 
We use the notation $B(u) = B(u,u)$, for simplicity. Motivated by the linearization of the Navier--Stokes equations we introduce the operator $\widetilde{B}:W(I)\times W(I)\to \mathcal{L}(W(I),L^2(I;V^*))$ defined as
 \begin{align*}
		\widetilde{B}(y_1,y_2)z = B(y_1,z) + B(z,y_2) \quad \forall y_1,y_2,z\in W(I).
\end{align*}  
Given $\overline{y}\in W(I)$, we also see that the Fr{\'e}chet derivative $B'(\overline{y})$ of the operator $B$ at $\overline{y}$ is defined as 
 \begin{align*}
		\langle B'(\overline{y})z, v  \rangle_{L^2(V^*),L^2(V)} = \langle \widetilde{B}(\overline{y},\overline{y})z, v  \rangle_{L^2(V^*),L^2(V)}\quad \forall z\in W(I),\, v\in L^2(I;V^*).
\end{align*} 
The adjoint of the linear operator $\widetilde{B}(\overline{y}_1,\overline{y}_2)$, for $\overline{y}_1,\overline{y}_2\in W(I)$, is denoted as $\widetilde{B}(\overline{y}_1,\overline{y}_2)^*$ and is defined by $$\langle \widetilde{B}(\overline{y}_1,\overline{y}_2)^*w, v  \rangle_{L^2(V^*),L^2(V)} = \langle B(\overline{y}_1,v), w  \rangle_{L^2(V)} + \langle B(v,\overline{y}_2), w  \rangle_{L^2(V)}.$$
We also see that the adjoint of the linear operator $B'(\overline{y})$ given $\overline{y}$ can be written as 
 \begin{align*}
		B'(\overline{y})^*w = \widetilde{B}(\overline{y},\overline{y})^*w \quad \forall w\in W(I).
\end{align*}


\section{Analysis of the governing equations}\label{sec:3}
In this section, we shall discuss some known well-posedness results for the Navier--Stokes equations and  an Oseen equation (linearization of the Navier--Stokes equations), as well as its corresponding adjoint system.

\vspace{.1 in}\noindent{\bf Navier--Stokes equations.} Given $y_0 \in H$ and $u\in L^2(I;V^*)$, we say that $y\in W(I)$ is a weak solution to the Navier--Stokes equations whenever it satisfies
 \begin{align}
		\langle \partial_t y(t), v  \rangle_{V^*,V} + \nu (\nabla y(t), \nabla v)_\Omega + b(u(t),u(t),v) = \langle u(t), v  \rangle_{V^*,V}\quad\forall v\in V,
		\label{weakvar:NS}
\end{align} 
for a.e. $t\in I$, and $y(0) = y_0$ in $H$. The evaluation $y(0)$ is well-defined due to the embedding $W(I)\hookrightarrow C(\overline{I};H)$.

The existence and regularity of solutions has been well studied, for such results we refer to  {\cite[Theorems III.3.1 and III.3.2]{temam1984}}.
\begin{theorem}
	Suppose that $u\in L^2(I;V^*)$ and $y_0\in H$, then there exists a unique $y \in W(I)$ such that \eqref{weakvar:NS} holds. Furthermore, the energy estimate
	 \begin{align*}
			\|y\|_{W(I)} \le c ( \|u\|_{L^2(I;V^*)} +\|y_0\|_H )
	\end{align*} 
	holds  for some constant $c>0$ independent of $u$ and $y_0$.
	\label{theorem:weakexistenceNS}
\end{theorem}
The norm $\|\cdot\|_{W(I)}$ majorizes the norms $\|\cdot\|_{L^\infty(I;H)}$ and $\|\cdot\|_{L^2(I;V)}$. The proof of \Cref{theorem:weakexistenceNS} can be carried out using the usual Galerkin methods, it is even possible to prove that the mapping $[(u,y_0)\mapsto y]: L^2(I;V^*)\times H \to W(I)$ is locally Lipschitz continuous.  {See \cite[p. 294-295]{temam1984} for more details.}

In the analysis of the optimal control problem, we will employ $L^s$-strong solutions of the Navier--Stokes equations, $s\in[2,\infty)$.  {We say that a weak solution $y\in W(I)$ is an $L^s$-strong solution of the Navier--Stokes equations \eqref{dynsys} if $\partial_ty \in L^s(Q)^2$ and $y \in L^s(I;W^{2,s}(\Omega)^2)$.}  The following theorem summarizes the existence of strong solutions for \eqref{dynsys} as well as the recovery of the pressure term, if we assume additional regularity on the initial data and the external force.  {We refer the reader to \cite[Theorem 2]{gerhardt1978} for the proof.} In fact, we have that $y\in W^{2,1}_s$, this is due to the fact that -- by following the proof of \cite[Theorem 2]{gerhardt1978} -- $y\in L^s(I;V_s)$. We also note that due to the embedding $W^{2,1}_s \hookrightarrow   C(\overline{I};W^{2-2/s,s}_{0,\sigma}(\Omega)^2)$, the appropriate space for the initial data is $W^{2-2/s,s}_{0,\sigma}(\Omega)^2$.
\begin{theorem}
	Let $s\ge 2$, the assumptions $u\in L^s(Q)^2$ and $y_0\in W_{0,\sigma}^{2-2/s,s}(\Omega)^2$ imply the unique existence of the strong solution $y \in W^{2,1}_s$ of \eqref{dynsys}. Furthermore, there exists  {$p\in L^s(I;W^{1,s}(\Omega))$, which is unique up to an element of $L^s(0,T)$,} such that
	 \begin{align}
			\left\{
			\begin{aligned}
				\partial_t y + \nu A y + B y + \nabla p & = u &&\text{in }L^s(I;H_s),\\
				y(0) & = y_0 &&\text{in }W_0^{2-2/s,s}(\Omega)^2,
			\end{aligned}
			\right.\label{strongwp:NS}
	\end{align} 
	and the pair $(y,p) \in W^{2,1}_s\times L^s(I;W^{1,s}(\Omega))$ satisfies the estimate
	 \begin{align}
			\|y\|_{W^{2,1}_s} + \|\nabla p \|_{L^s({Q})^2} \le c ( \|y_0\|_{L^s(\Omega)^2} + \|u\|_{L^s(Q)^2} ).
			\label{energy:strongLP}
	\end{align} 
	\label{existence:strongNS}
\end{theorem}

We also have a weak-strong convergence for the force-to-velocity operator.
\begin{theorem}
	\label{continuity:control-to-state}
	Let $\{ u_k \}\subset L^s(Q)^2$ be a sequence converging weakly to $u\in L^s(Q)^2$, where $s>2$. Then $y_k\to y$ in $C(\overline{Q})^2$, where $y_k$ and $y$ solve \eqref{strongwp:NS} with $u_k$ and $u$ as the external forces, respectively.
\end{theorem}
 {\begin{proof}
		From \eqref{energy:strongLP} and the compact embedding $W^{2,1}_s \hookrightarrow C(\overline{Q})^2$ we see that, up to a subsequence, $y_k\rightharpoonup y$ in $W^{2,1}_s$ and  $y_k\to y$ in $C(\overline{Q})^2$. A routinary \textit{passage to the limit} argument shows that $y\in W^{2,1}_s$ solves \eqref{strongwp:NS} with $u\in L^s(Q)^2$ as the right-hand side.    
\end{proof}}

\vspace{.1 in}\noindent{\bf Oseen equations.} As usual, in  optimal control problems,  one needs to consider the linearization of the systems governing the states. For this reason, for known elements $\overline{y}_1, \overline{y}_2$ and external force $v$,  we shall consider the Oseen equations
 \begin{align}\label{system:linearnavierstokes}
		\left\{ \begin{aligned}
			\partial_tz - \nu\Delta z + (\overline{y}_1\cdot\nabla)z + (z\cdot\nabla)\overline{y}_2  + \nabla q = v\,\,\,\,\text{in }Q,
			\\ \dive {z}=0\,\,\,\,\text{in }Q,\,\,  {z}=0\,\,\,\,\text{on }\Sigma,\,\,\, {z}(\cdot,0)= z_0\,\,\,\text{in }\Omega.
		\end{aligned}
		\right.
\end{align} 
We can write the weak formulation of \eqref{system:linearnavierstokes} as 
 \begin{align}
		\left\{
		\begin{aligned}
			\partial_t z + \nu Az + \widetilde{B}(\overline{y}_1,\overline{y}_2)z &= v &&\text{in }L^2(I;V^*),\\
			z(0) & = z_0 &&\text{in }H.
		\end{aligned}\right.\label{weakop:linear}
\end{align} 

The existence of weak solutions of the linearized system can be proven using the usual Galerkin method just as in the nonlinear system.

\begin{theorem}[ {\cite[Proposition 2.1]{hinze2002}}]
	Suppose that $v\in L^2(I;V^*)$, $z_0\in H$ and $\overline{y}_1,\overline{y}_2\in L^2(I;V)\cap L^\infty(I;H)$, then there exists a unique $z \in W(I)$ such that \eqref{weakop:linear} holds. Furthermore, the energy estimate
	 \begin{align*}
			\|z\|_{W(I)} \le c \left( \|z_0\|_H +  \|v\|_{L^2(I;V^*)} \right)
	\end{align*} 
	holds for some constant $c>0$ independent of $z_0$ and $v$.
	\label{theorem:weakexistencelinearNS}
\end{theorem}

Quite similarly as in the nonlinear case, improved regularity of the external force and initial datum leads to a more regular solution. 
\begin{theorem}[ {\cite[Theorem 4.1.33]{dwachsmuth2006}}]
	Let $s\ge 2$. The assumptions $v\in L^s(Q)^2$, $z_0\in W_{0,\sigma}^{2-2/s,s}(\Omega)^2$ and $\overline{y}_1,\overline{y}_2 \in W^{2,1}_s$ imply the unique existence of the strong solution $z \in W^{2,1}_s$ of \eqref{system:linearnavierstokes}. The existence of a pressure term $q\in L^s(I;W^{1,s}(\Omega))$, which is unique up to an element of $L^s(0,T)$, is also guaranteed and is known to satisfy 
	 \begin{align}
			\left\{
			\begin{aligned}
				\partial_t z + \nu Az + \widetilde{B}(\overline{y}_1,\overline{y}_2)z + \nabla q &= v &&\text{in }L^s(I;H_s),\\
				z(0) & = z_0 &&\text{in }W_{0,\sigma}^{2-2/s,s}(\Omega)^2.
			\end{aligned}\right.\label{strongwq:linear}
	\end{align} 
	Furthermore, the pair $(z,q)\in W^{2,1}_s\times L^s(I;W^{1,s}(\Omega)\cap L^s(\Omega)/\mathbb{R})$ satisfies
	 \begin{align}
			\|z\|_{W^{2,1}_s} + \|\nabla q\|_{L^s(Q)^2} \le c (\| z_0\|_{W_{0,\sigma}^{2-2/s,s}(\Omega)^2} + \|v\|_{L^s(Q)^2}).
			\label{energy:stronglinearLP}
	\end{align} \label{theorem:stronglinear}
\end{theorem}

\vspace{.1 in}\noindent{\bf Dual Oseen equations.} Given elements $\overline{y}_1,\overline{y}_2\in W(I)$ and $v\in L^\alpha(I;V^*)$ we are also interested in looking at the following system, which is also known as the {\it adjoint} to the Oseen equations introduced above 
 \begin{align*}
		\left\{ \begin{aligned}
			-\partial_t w - \nu \Delta w - (\overline{y}_2\cdot\nabla)w + (\nabla \overline{y}_2)^\top w + \nabla r = v\,\,\,\,\text{in Q},
			\\ \dive {w}=0\,\,\,\,\text{in Q},\,\,  {w}=0\,\,\,\,\text{on $\Sigma$},\,\,\, {w}(\cdot,T)= 0\,\,\,\text{in $\Omega$}.
		\end{aligned}
		\right.
\end{align*} 
The variational form of the system above is 
 \begin{align}
		\begin{aligned}
			-\partial_t w + \nu Aw + \widetilde{B}(\overline{y}_1,\overline{y}_2)^*w & = v &&\text{in }L^\alpha(I;V^*)\\
			w(T) & = 0 &&\text{in }H.
		\end{aligned}\label{weakop:adjoint}
\end{align} 

The regularity of weak solutions of the adjoint equations are not the same as that of \eqref{weakvar:NS} and \eqref{weakop:linear}. In fact, we get $W^{4/3}(I)$ instead of having the solutions belong to $W(I)$. We have such a result summarized in the following theorem.
\begin{theorem}[ {\cite[Proposition 2.1]{hinze2002}}]
	Let $v\in L^{4/3}(I;V^*)$, and $\overline{y}_1,\overline{y}_2\in W(I)$. Then a solution $w\in W^{4/3}(I)$ of \eqref{weakop:adjoint} exists and satisfies
	 \begin{align*}
			\|w\|_{W^{4/3}(I)} \le c\|v\|_{L^{4/3}(I;V^*)}. 
	\end{align*} 
	for some constant $c>0$ independent of $v$.
	\label{theorem:weakadjoint}
\end{theorem}
Despite the lack of time regularity of the weak solution, we  could eventually recover uniqueness and $L^s$-regularity of the adjoint solution given additional regularity of the data.
\begin{theorem}[ {\cite[Theorem 4.1.35]{dwachsmuth2006}}]
	Suppose that $v\in L^s(Q)^2$ with $s\ge 2$, and $\overline{y}_1,\overline{y}_2\in {W}^{2,1}_s$. The weak solution of \eqref{weakop:adjoint} is unique and is an $L^s$-strong solution and belongs to $W^{2,1}_s$. Furthermore, there exists $r\in L^s(I;W^{1,s}(\Omega))$, which is unique up to an element of $L^s(0,T)$, such that
	 \begin{align}
			\begin{aligned}
				-\partial_t w + \nu Aw + \widetilde{B}(\overline{y}_1,\overline{y}_2)^*w + \nabla r & = v &&\text{in }L^s(I;H_s)\\
				w(T) & = 0 &&\text{in }W^{2-2/s,s}_0(\Omega)^2.
			\end{aligned}\label{strongwr:adjoint}
	\end{align} 
	The pair $(w,r)\in {W}^{2,1}_s\times L^s(I;W^{1,s}(\Omega)\cap L^s(\Omega)/\mathbb{R}) $, furthermore, satisfies the energy inequality
	 \begin{align*}
			\|w\|_{W^{2,1}_s} + \|\nabla r\|_{L^s(Q)^2} \le c\|v\|_{L^{s}(Q)^2},
	\end{align*} 
	from some constant $c>0$ independent of $v$.\label{th:stradjoint}
\end{theorem}

We end this section with the following lemma which was inspired by the $L^s$-$L^1$ stability of solutions proven in \cite[Lemma 2.3]{casas2022a} and \cite[Lemma 2]{corella2023}. 
\begin{lemma}
	Let $v\in L^{s}(Q)^2$ for $s>2$, and $\overline{y} \in W^{2,1}_s$. If $z_v \in W^{2,1}_s$ is the strong solution of \eqref{strongwq:linear} and $w_v \in W^{2,1}_s$ is the solution of \eqref{strongwr:adjoint} both with $\overline{y}_1=\overline{y}_2=\overline{y}$.
	Then for any $\tilde{s}\in [1,2)$, there exists $c>0$ independent on $v$ such that 
	 \begin{align}
			\max\{ \|z_v\|_{L^{\tilde{s}}(Q)^2},\|w_v\|_{L^{\tilde{s}}(Q)^2} \} \le c\|v\|_{L^1(Q)^2}.\label{estimate:lsl1}
	\end{align} 
	\label{lemma:lsl1}
\end{lemma}
\begin{proof}
	Since $v\in L^s(Q)^2$ for $s>2$, $z_v \in C(\overline{\Omega})^2$. This implies that $|z_v|^{\tilde{s}-2}z_v\in L^{\tilde{s}'}(Q)^2,$ where $\tilde{s}' = \tilde{s}/(\tilde{s}-1) > 2$.
	The solution $\mathfrak{w}$ satisfying
	 \begin{align*}
			\begin{aligned}
				-\partial_t \mathfrak{w} + \nu A \mathfrak{w} + B'(\overline{y})^*\mathfrak{w} + \nabla \mathfrak{r} & = |z_v|^{\tilde{s}-2}z_v &&\text{in }L^{\tilde{s}'}(Q)^2\\
				\mathfrak{w}(T)& = 0 &&\text{in }W_0^{2-2/{\tilde{s}'},\tilde{s}'}(\Omega)^2,
			\end{aligned}
	\end{align*} 
	belongs to $C(\overline{Q})^2$, and satisfies --- according to \Cref{th:stradjoint} ---
	 \begin{align*}
			\begin{aligned}
				\|\mathfrak{w}\|_{C(\overline{Q})^2} & \le c \||z_v|^{\tilde{s}-2}z_v\|_{L^{ \tilde{s}'}(Q)^2} \le c\left( \int_Q \left||z_v|^{\tilde{s}-2}z_v \right|^{\tilde{s}'} \du x\du t \right)^{1/\tilde{s}'}\\
				& \le c\left( \int_Q |z_v|^{\tilde{s}} \du x\du t \right)^{(\tilde{s}-1)/\tilde{s}}\le c\|z_v\|^{\tilde{s}-1}_{L^{\tilde{s}}(Q)^2}.
			\end{aligned}
	\end{align*} 
	Using integration by parts, and the definition of the adjoint of the operators, we get
	 \begin{align*}
			\begin{aligned}
				& \|z_v\|^{\tilde{s}}_{L^{\tilde{s}}(Q)^2}  = \int_{Q} |z_v|^{\tilde{s}} \du x\du t = \big( |z_v|^{\tilde{s}-2}z_v,z_v \big)_Q  = \big( -\partial_t \mathfrak{w} + \nu A \mathfrak{w} + B'(\overline{y})^*\mathfrak{w} + \nabla\mathfrak{r},z_v \big)_Q \\
				& = \big( \mathfrak{w}, \partial_t z_v + \nu Az_v + B'(\overline{y})z_v + \nabla {q} \big)_Q = \big( \mathfrak{w}, v \big)_Q \le \|v\|_{L^1(Q)^2}\|\mathfrak{w}\|_{C(\overline{Q})^2}\\
				& \le c\|v\|_{L^1(Q)^2}\|z_u\|^{\tilde{s}-1}_{L^{\tilde{s}}(Q)^2}.
			\end{aligned}
	\end{align*} 
	The transition from the second line to the third line in the computation above used the fact that $\nabla\cdot z_v = \nabla\cdot\mathfrak{w} = 0$. This cancels out the term $\mathfrak{r}\nabla\cdot z_v$ and allows the addition of $q\nabla\cdot\mathfrak{w}$. We thus have the estimate for $z_v$ in \eqref{estimate:lsl1}. 
	
	For the solution $w_v \in W^{2,1}_s$, we also get $w_v \in C(\overline{Q})^2$ from which we get $|w_v|^{\tilde{s}-2}w_v\in L^{\tilde{s}'}(\Omega)^2$ whence the solution $\mathfrak{Z}\in C(\overline{Q})^2$ to the equations
	 \begin{align*}
			\begin{aligned}
				\partial_t \mathfrak{Z} + \nu A \mathfrak{Z} + B'(\overline{y})\mathfrak{Z} + \nabla\mathfrak{q} & = |w_v|^{\bar{s}-2}w_v &&\text{in }L^{\tilde{s}'}(Q)^2\\
				\mathfrak{Z}(0)& = 0 &&\text{in }W_0^{2-2/{\tilde{s}'},\tilde{s}'}(\Omega)^2,
			\end{aligned}
	\end{align*} 
	satisfies --- by virtue of \Cref{theorem:stronglinear} --- the estimate $\|\mathfrak{Z}\|_{C(\overline{Q})^2} \le c \|w_v\|^{\tilde{s}-1}_{L^{\tilde{s}}(Q)^2}$. From this, we get the estimate
	 \begin{align*}
			\begin{aligned}
				&\|w_v\|_{L^{\tilde{s}}(Q)^2}^{\tilde{s}}  = \int_{Q} |w_v|^{\tilde{s}} \du x\du t = \big(|w_v|^{\tilde{s}-2}w_v,w_v \big)_Q  = \big( \partial_t \mathfrak{Z} + \nu A \mathfrak{Z} + B'(\overline{y})\mathfrak{Z} + \nabla\mathfrak{q},w_v \big)_Q\\
				& = \big( \mathfrak{Z},-\partial_t w_v + \nu A w_v + B'(\overline{y})^*w_v + \nabla r \big)_Q  = \big( \mathfrak{Z},v \big)_Q \le \|\mathfrak{Z}\|_{C(\overline{Q})^2}\|v\|_{L^1(Q)^2}\\
				& \le c\|w_v\|^{\tilde{s}-1}_{L^{\tilde{s}}(Q)^2}\|v\|_{L^1(Q)^2}.
			\end{aligned}
	\end{align*} 
	This finally establishes estimate \eqref{estimate:lsl1}.
\end{proof}


\section{The velocity tracking problem}\label{sec:4}
In this section, we give a concise statement of the optimal control problem and specify the data  assumptions. We also recall some known results including the first-order necessary conditions and the second-order sufficient condition.

We fix positive numbers $\bar s$ and $\bar s'$ such that  $\bar s>2$ and $1/\bar s+1/\bar s'=1$. For the optimal control problem \eqref{opticon}, we will make the following standing assumption. 

\begin{assumption}
	 {The following assumptions hold}:
	\begin{itemize}\item[(i)] the initial datum $y_0$ belongs to $W^{2-\frac{2}{\bar s},\bar s}_{0.\sigma}(\Omega)^2$;
		
		\item[(ii)] the tracking datum $y_d$ belongs to $L^{\infty}(Q)^2$;
		
		\item[(iii)] the functions $u_a, u_b:Q\to\mathbb R^2$ are bounded.
	\end{itemize}
\end{assumption}

We recall that the control set is given by
 \begin{align*}
		\mathcal U=\left\lbrace  u = (u^1,u^2)\in  L^\infty(Q)^2:\,\,  u^j_a\le u^j\le  u_b^j\,\,\,\,\text{a.e. in $Q$ for each $j=1,2$} \right\rbrace.
\end{align*} 
By the box constraints imposed on the controls, $\mathcal{U}$  {is bounded in $L^\infty(Q)$. From this} we define
 \begin{align}\label{uniformbound:control}
		M_{\mathcal U}:=\sup_{u\in\mathcal U}|u|_{L^\infty(Q)^2}.
\end{align} 

Denoting the \textit{objective functional} by $\mathcal J:L^{\bar{s}}(Q)^2\to \mathbb R$, i.e., 
 \begin{align*}
		\mathcal J(u):=\frac{1}{2}\int_{0}^T\int_{\Omega}\, |{y_u}(t,x)- y_d(t,x)|^2\, dx\,dt,
\end{align*} 
 {we consider the optimal control problem }
 \begin{align}
		\min_{u\in\mathcal{U}}\mathcal J(u)\quad\text{ subject to } \eqref{dynsys}.
		\label{opticon}\tag{P}
\end{align} 
\begin{definition}
	Let $ \bar u\in\mathcal U$. We define the minimality radius of $ \bar u$ as
	 \begin{align*}
			\bar r_{ \bar u}:=\sup\left\lbrace \delta\ge0:\,\,\mathcal J( \bar u)\le\mathcal J( u)\,\,\,\,\text{for all $ u\in\mathcal U$ with $\| u- \bar u\|_{{L^{\bar{s}}(Q)^2}}\le\delta$}  \right\rbrace.
	\end{align*} 
	We say that $ \bar u$ is a local minimizer of problem \eqref{opticon} if $\bar r_{ \bar u}<+\infty$ and that $ \bar u$ is a global minimizer of problem \eqref{opticon} if $\bar r_{ \bar u}=+\infty$.
\end{definition}

The existence of at least one global minimizer for problem \eqref{opticon}  follows from the weak sequential lower semi-continuity of the objective functional and the continuity of the force-velocity operator according to \Cref{continuity:control-to-state}.

\vspace{.1 in}\noindent{\bf Force-to-velocity mapping.}
By \Cref{existence:strongNS}, to each control and each initial datum corresponds a unique state. The mapping $\widetilde{\mathcal{S}}:L^{\bar{s}}(Q)^2\times W^{2-2/\bar{s},2}_{0,\sigma} \to W_{\bar s}^{2,1}$ given by $\widetilde{\mathcal{S}}(u,y_0)=y_u$ is called \textit{the data-to-velocity mapping}. And for a fixed initial datum $y_0\in W^{2-2/\bar{s},2}_{0,\sigma}$ we define $\mathcal{S}:L^{\bar{s}}(Q)^2 \to W_{\bar s}^{2,1}$ as $\mathcal{S}(\cdot):= \widetilde{\mathcal{S}}(\cdot,y_0)$ which we call the \textit{force-to-velocity} map. Such mapping is infinitely differentiable.
\begin{theorem}[{\cite[Theorem 2.5]{CKbangbang}}]
	\label{thmctso}
	The force-to-velocity mapping $\mathcal S:L^{\bar{s}}(Q)^2\to W_{\bar s}^{2,1}$ is of class $C^\infty$, i.e., it has  Fr\'echet derivatives of all orders. Moreover, the following statements hold:
	\begin{itemize}
		\item[(i)] The first  {order} derivative is given by 
		 \begin{align*}
				\mathcal S'(u)v=z_{u,v}\quad \text{for } u,v\in L^{\bar s}(Q)^2,
		\end{align*} 
		where each $z_{u,v}\in W_{\bar s}^{2,1}$ is the unique $L^{\bar s}$-strong solution of the Oseen equations \eqref{system:linearnavierstokes} with $\overline{y}_1 = \overline{y}_2 = y_u$ and zero initial datum.  
		
		\item[(ii)] The second  {order} derivative is given by 
		 \begin{align*}
				\mathcal S''(u)\big(v_1,v_2\big)=z_{u,(v_1,v_2)}\quad \text{for } u,v_1,v_2\in L^{\bar s}(Q)^2,
		\end{align*} 
		where each $z_{u,(v_1,v_2)}\in W_{\bar s}^{2,1}$ is the unique $L^{\bar s}$-strong solution of the system 
		 \begin{align*}
				\begin{aligned}
					\partial_t z + \nu Az + B'(y)z + \nabla q&= -B''(y_u)[z_{u,v_1},z_{u,v_2}] &&\text{in }L^{\bar s}(Q)^2\\
					z(0)& = 0&&\text{in } W_{0,\sigma}^{2-2/\bar s,\bar s}(\Omega)^2.
				\end{aligned}
		\end{align*} 
		where $z_{u,v_1} = \mathcal{S}'(u)v_1$ and $z_{u,v_2} = \mathcal{S}'(u)v_2$.
	\end{itemize}

\end{theorem}

 {
	Since $\mathcal{U} \subset L^{\bar{s}}(Q)^2$ the derivatives are also well-defined on $u\in \mathcal{U}$ in the directions $v\in L^{\bar{s}}(Q)^2$ and $(v_1,v_2)\in L^{\bar{s}}(Q)^2\times L^{\bar{s}}(Q)^2$ with $u+v, u+v_1,u+v_2 \in \mathcal{U} $, see also the first paragraph after Theorem 2.3 in \cite{Vtyt1}.
}

\vspace{.1 in}\noindent{\bf Force-to-covelocity mapping.}
Since $\mathcal{S}'(u) \in \mathcal{L}(L^{\bar{s}}(Q)^2;W^{2,1}_{\bar{s},0})$, we have its adjoint denoted as $\mathcal{S}'(u)^*$. For an arbitrary force $\bar{v}\in L^{\bar{s}}(Q)^2$ the \textit{costate}  $w_{u,\bar{v}} = \mathcal{S}'(u)^*\bar{v}$ is the element that solves \eqref{weakop:adjoint} with $\overline{y}_1=\overline{y}_2 = y_u$ and $\bar{v}$ as the right-hand side under the appropriate spaces.

To this end, we recall some stability estimates for the maps $\mathcal{S}$, its linearization $\mathcal{S}'$ and the adjoint $\mathcal{S}'(\cdot)^*$.

\begin{lemma}[{\cite[Corollary 2.6]{CKbangbang}}]
	Let $u,\bar u\in L^{\bar{s}}(Q)^2$, $y_u = \mathcal{S}(u)\in W^{2,1}_{\bar{s}}$ and $y_{\bar{u}} = \mathcal{S}(\bar{u})\in W^{2,1}_{\bar{s}}$ be the states corresponding to $u$ and $\bar{u}$, respectively. Then there exists a constant $c>0$ such that 
	 \begin{align}
			\| y_u - y_{\bar{u}} \|_{W^{2,1}_{\bar{s}}} \le c \|u - \bar u\|_{L^{\bar{s}}(Q)^2}.
			\label{state-to-controlgap}
	\end{align} 
	\label{lemma:state-to-controlgap}
\end{lemma}

\begin{lemma}[{\cite[Lemma 3.8]{CKbangbang}}]
Let $u,\bar u\in L^{\bar{s}}(Q)^2$, and $y_u = \mathcal{S}(u)\in W^{2,1}_{\bar{s}}$ and $y_{\bar{u}} = \mathcal{S}(\bar{u})\in W^{2,1}_{\bar{s}}$ be the states respectively corresponding to $u$ and $\bar{u}$, and $z_{\bar u, u - \bar u} = S'(\bar u)(u - \bar u)\in W^{2,1}_{\bar{s}}$ be the linearized state corresponding to the Fr{\'e}chet derivative of $\mathcal{S}$ at $\bar u$ in the direction $u-\bar u$. There exists $\delta>0$ such that whenever $\|u-\bar{u}\|_{L^1({Q})^2}\le \delta$ we have
 \begin{align*}
	\| y_u - y_{\bar u} \|_{L^2(Q)^2} \le 2\|z_{\bar u, u - \bar u} \|_{L^2(Q)^2} \le 3\|y_u - y_{\bar u} \|_{L^2(Q)^2}.
\end{align*} 
\label{lemma:gap-to-linear}
\end{lemma}

\begin{lemma}[{\cite[Lemma 3.9]{CKbangbang}}]
Let $v,u,\bar{u}\in L^{\bar{s}}(Q)^2$, setting $y_{u} = \mathcal{S}(u), y_{\bar{u}} = \mathcal{S}(\bar{u}), z_{u,v} = \mathcal{S}'(u)v, z_{\bar{u},v} = \mathcal{S}'(\bar{u})v $ one finds constants $c>0$ such that the following estimates hold
 \begin{align}
	& \max\{\|z_{u,v} - z_{\bar{u},v}\|_{L^2(Q)^2}, \|\nabla(z_{u,v} - z_{\bar{u},v})\|_{L^2(Q)^2} \}\le  c\|y_u - y_{\bar{u}}\|_{C({\overline{Q}})^2}\| z_{\bar{u},v}\|_{L^2(Q)^2}
	\label{estimate:linear-controlnstate1}
\end{align} 
Furthermore, there exists $\delta>0$ such that 
 \begin{align}
	\|z_{\bar{u},v}\|_{L^2(Q)^2} \le 2\|z_{u,v}\|_{L^2(Q)^2} \le 3\|z_{\bar{u},v}\|_{L^2(Q)^2}
	\label{estimate:linear-linear}
\end{align} 
whenever $\|u - {\bar{u}}\|_{L^1({Q})^2}<\delta$.
\label{lemma:lineargap-linear}
\end{lemma}

\begin{remark}
We note that the lemmata above have been proven analogously in \cite{CKbangbang}, but of course with a few changes on the proofs. We decided not to write the full proofs but we mention a few tweaks from the proofs presented in the said reference to arrive in their current states:
\begin{itemize}
\item[$\bullet$] Lemma \ref{lemma:state-to-controlgap}: one can skip using the embedding $W_{\bar{s}}^{2,1}\hookrightarrow C(\overline{Q})^2$ and can stop before using the inequality $\|u - \bar u\|_{L^{\bar{s}}(Q)^2}\le M_{\mathcal{U}}\|u - \bar u\|_{L^2(Q)^2}^{2/{\bar{s}}}$.
\item[$\bullet$] Lemma \ref{lemma:gap-to-linear}: choose $\varepsilon = \frac{1}{2\bar M\sqrt{T}}$ instead of that one presented in \cite[Lemma 3.8]{CKbangbang}.
\item[$\bullet$] Lemma \ref{lemma:lineargap-linear}: one can control the gap of the states by controlling the gap of the controls to achieve \eqref{estimate:linear-linear} from inequality (3.24) in \cite{CKbangbang}.  {We also mention that in the proof \cite[Lemma 3.9]{CKbangbang}, the authors extracted the term $y_u - y_{\bar{u}}$ from the trilinear form via its boundedness in $L^\infty(Q)^2$, but we point out that --- because of the embedding $W^{2,1}_{\bar{s}}\hookrightarrow C(\overline{Q})^2$ --- we have the liberty to use the supremum norm instead of the essential supremum.}
\end{itemize}
\end{remark}

\begin{lemma}\label{adjointgap}
Let $u,\bar{u}\in L^{\bar{s}}(\Omega)^2$, $y_u = \mathcal{S}(u), y_{\bar{u}} = \mathcal{S}(\bar{u}), w_u = \mathcal{S}'(u)^*(y_u-y_d),w_{\bar{u}} = \mathcal{S}'(\bar{u})^*(y_{\bar{u}}-y_d)\in W^{2,1}_{\bar{s}}$. There exists $c>0$ such that the following inequality hold
 \begin{align}
	&\begin{aligned}
		\|w_u - w_{\bar{u}}\|_{C(\overline I;C^1(\overline{\Omega})^2)} \le  c(1+\|y_{\bar{u}} - y_{d} \|_{L^s({Q})^2})\|y_u - y_{\bar{u}} \|_{C(\overline{Q})^2},
	\end{aligned}\label{estimate:adjdiff2}
\end{align} 
for arbitrary $s>4$.
\end{lemma}

\begin{proof}
 {From the assumptions $y_u - y_d, y_{\bar{u}} - y_d \in L^s(Q)^2$ for any $s>4$, because $y_u,y_{\bar{u}}\in C(\overline{Q})^2$ and $y_d\in L^\infty(Q)^2$. This implies --- by virtue of \Cref{th:stradjoint} and the embedding $W^{2,1}_s\hookrightarrow C(\overline{I};C^1(\overline{Q})^2)$ for $s>4$ --- that $w_u = \mathcal{S}'(u)^*(y_u-y_d)$ satisfies
 \begin{align}\label{cc2:p_v}
		\|w_{\bar{u}}\|_{C(\overline{I};C^1(\overline{\Omega})^2)} \le c \|y_{\bar{u}}-y_d \|_{L^s(Q)^2}.
\end{align} }

 {Since the element $\mathfrak{w}:= w_u-w_{\bar{u}}\in W^{2,1}_s$ solves
 \begin{align*}
		\begin{aligned}
			-\partial_t\mathfrak{w} + \nu A\mathfrak{w} + B'(y_u)^*\mathfrak{w} + \nabla\mathfrak{r} & = (y_u-y_{\bar{u}}) - B'(y_u-y_{\bar{u}})^*w_{\bar{u}} && \text{in }L^s(Q)^2,\\
			\mathfrak{w}(T) & = 0 &&\text{in }W^{2-2/s,s}_0(\Omega)^2,
		\end{aligned}
\end{align*} 
we see, from \Cref{th:stradjoint}, that 
 \begin{align}
		\begin{aligned}
			\|\mathfrak{w}\|_{C(I;C^1(\overline{\Omega}))} \le&\, 
			\left\|y_u-y_{\bar{u}}\right\|_{L^s(Q)^2} + \left\| B'(y_u-y_{\bar{u}})^*w_{\bar{u}}\right\|_{L^s(Q)^2}\\
			\le&\, |\Omega|^{1/s}\left\|y_u-y_{\bar{u}}\right\|_{C(\overline{Q})^2} + \left\| B'(y_u-y_{\bar{u}})^*w_{\bar{u}}\right\|_{L^s(Q)^2}
		\end{aligned}\label{initestimate:adjdiff}
\end{align} 
The second term on \eqref{initestimate:adjdiff} is estimated by taking an arbitrary element $v\in L^{s'}(Q)^2$ where $s'$ is the H{\"o}lder conjugate of $s$, and majorizing $\int_0^T\langle B'(y_u-y_{\bar{u}})^*w_{\bar{u}},v\rangle\du t$. H{\"o}lder inequality, and estimates \eqref{cc2:p_v} and \eqref{state-to-controlgap} thus yield
 \begin{align*}
		\begin{aligned}
			&\left|\langle B'(y_u-y_{\bar{u}})^*w_{\bar{u}},v\rangle_{L^2(V)} \right| = \left| (((y_u-y_{\bar{u}})\cdot\nabla) w_{\bar{u}}, v)_Q + ((v\cdot\nabla)w_{\bar{u}}, (y_u-y_{\bar{u}}))_Q \right|\\
			& \le 2c\|y_u-y_{\bar{u}}\|_{L^s(Q)^2}\| w_{\bar{u}}\|_{C(I;C^1(\overline{\Omega}))}\|v \|_{L^{s'}(Q)^2}\\
			& \le 2c|\Omega|^{1/s} \|y_{\bar{u}}-y_d \|_{L^s(Q)^2}\|y_u-y_{\bar{u}}\|_{C(\overline{Q})^2}\|v \|_{L^{s'}(Q)^2}
		\end{aligned}
\end{align*} 
We thus get \eqref{estimate:adjdiff2} by taking $c:=|\Omega|^{1/s} \max\{1, 2c\}$.}  
\end{proof}

 {\begin{remark}
If $\bar{u}\in \mathcal{U}$, we can use \eqref{uniformbound:control} so that 
 \begin{align}
		&\begin{aligned}
			\|w_u - w_{\bar{u}}\|_{C(\overline I;C^1(\overline{\Omega})^2)} \le  c\|y_u - y_{\bar{u}} \|_{C(\overline{Q})^2},
		\end{aligned}\label{estimate:adjdiff1}
\end{align} 
for some constant $c>0$ independent of $u$ and $\bar{u}$.
\end{remark}}

\vspace{.1 in}\noindent{\bf First-order necessary condition.}
Let us begin recalling the first and second variations of the objective functional.
\begin{theorem}\label{derJ}
The objective functional $\mathcal J:  {L^{\bar{s}}(\Omega)^2}\to \mathbb R$ is of class $C^\infty$. Moreover, the following statements hold. 
\begin{itemize}
\item[(i)] The first variation is given by 
 \begin{align*}
		\mathcal J'(u)v=\int_Q w_{u}\cdot v\, dx\, dt\quad\text{for } u,v\in L^{\bar s}(Q)^2,
\end{align*} 
where $w_u: = \mathcal{S}'(u)^*(y_u - y_d)\in W_{\bar s}^{2,1}$.

\item[(i)] The second variation  is given by 
 \begin{align*}
		\mathcal J''(u)v^2=\int_Q \Big[|z_{u,v}|^2-2\big(z_{u,v}\cdot \nabla\big) z_{u,v}\cdot w_u\Big]\, dx\, dt\quad\text{for } u,v\in L^{\bar s}(Q)^2.
\end{align*} \label{J:secondder}
\end{itemize}
where $z_{u,v} = \mathcal{S}'(u)v\in W^{2,1}_{\bar{s}}$.
\end{theorem}

We are now in a position to lay down the first-order necessary condition for problem \eqref{opticon}.
\begin{theorem}
Let $ \bar u\in\mathcal U$ be a local minimizer of problem \eqref{opticon}. Then
 \begin{align}\label{firstor}
	\int_Q w_{\bar u}\cdot\big(u-\bar u\big)\, dx\,dt\ge 0\quad\text{for all } u\in\mathcal U.
\end{align} 
\end{theorem}

We can immediately see that \eqref{firstor} is a direct consequence of $\mathcal{J}'(\bar u)(u-\bar{u}) \ge 0$ whenever $\bar{u}\in\mathcal{U}$ is a minimizer. Furthermore, from the first order necessary condition a local minimizer $\bar u=(\bar u_1,\bar u_2)$ must satisfy 
 \begin{align*}
\bar u^j(x,t)=\left\{ \begin{array}{l}
	\bar u_{a}^j(x,t)\,\,\,\,\text{if $w_{\bar u}^j(x,t)>0$}\\
	\\ \bar u_{b}^j(x,t)\,\,\,\,\text{if $w_{\bar u}^j(x,t)<0$},
\end{array}
\right.
\end{align*} 
for $j\in\{1,2\}$ and a.e. $(t,x )\in Q$. It is not hard to deduce from the first order necessary condition that a control $\bar u$ is bang-bang if and only if
 \begin{align*}
\meas\bigl\{(x,t)\in Q:\, w_{\bar u}^1(x,t)=0\,\,\,\text{or}\,\,\, w_{\bar u}^2(x,t)=0\bigr\}=0.
\end{align*} 
We recall that a control $\bar u$ is said to be bang-bang if $\bar u^j(x,t)\in\{u_{a}^j(x,t),u_{b}^j(x,t)\}$ for a.e. $(x,t)\in Q$.

\vspace{.1 in}\noindent{\bf A second-order sufficient condition.}	
In order to establish the second-order sufficient condition, we shall need the following lemma, which deals with the curvature of the objective functional. Aside from the sufficient condition the lemma below will be vital for the stability result in the upcoming section.
\begin{lemma}\label{curveest}
Let $\mu\in[1,2)$, $\theta \in [0,1]$ and $\bar u\in\mathcal U$. For every $\varepsilon>0$ there exists $\delta>0$ such that 
 \begin{align*}
|\mathcal J''(\bar{u} + \theta(u-\bar u))(u-\bar u)^2-\mathcal J''(\bar u)(u-\bar u)^2|\le \varepsilon \|u-\bar u\|_{L^{1}(Q)^2}^{\mu+1},
\end{align*} 
for all $u\in\mathcal U$ with $\|u-\bar u\|_{L^1(Q)^2}\le \delta$.
\end{lemma}
The proof is postponed to the end of this section but we now give the promised sufficient condition.
\begin{theorem}\label{theorem:suffsec}
Let $\mu\in[1,2)$ and $\bar u\in\mathcal U$.  Suppose that there exists $\delta>0$ such that
 \begin{align}\label{growthasunec}
\mathcal J'(\bar u)(u-\bar u)+\frac{1}{2}\mathcal J''(\bar u)(u-\bar u)^2\ge \|u-\bar u\|_{L^1(Q)^2}^{\mu+1}\quad \text{for all $u\in\mathcal U$ with $\|u-\bar u\|_{L^1(Q)^2}\le\delta$.}
\end{align} 
Then there exists $c>0$ such that
 \begin{align}\label{suffop}
\mathcal J(u)\ge \mathcal J(\bar u)+c\|u-\bar u\|_{L^1(Q)^2}^{\mu+1}\quad\text{for all $u\in\mathcal U$ with $\|u-\bar u\|_{L^1(Q)^2}\le \delta$}.
\end{align} 
In particular, $\bar u$ is a strict local minimizer.
\end{theorem}
\begin{proof}
Define $\bar u_\theta:=\bar u+\theta(u-\bar u)$ for $\theta\in [0,1]$. By the Taylor theorem, there exists a $\theta\in (0,1)$ such that 
 \begin{align*}
\mathcal J(u)-\mathcal J(\bar u)&=\mathcal J'(\bar u)(u-\bar u)+\frac{1}{2}\mathcal J''(\bar u_\theta)(u-\bar u)^2\\
&\geq \mathcal J'(\bar u)(u-\bar u)+\frac{1}{2}\mathcal J''(\bar u)(u-\bar u)^2-\frac{1}{2}\Big\vert \mathcal J''(\bar u)(u-\bar u)^2-\mathcal J''(\bar u_\theta)(u-\bar u)^2\Big\vert .
\end{align*} 
Now \eqref{suffop} follows from \eqref{growthasunec} and Lemma \ref{curveest}.  
\end{proof}

We conclude with the proof of Lemma \ref{curveest}.

\begin{proof}[Proof of Lemma \ref{curveest}]
Denote by $u_\theta = \bar{u} + \theta(u-\bar u)$, $y_\theta = \mathcal{S}(u_\theta)$, $y_{\bar{u}} = \mathcal{S}(\bar{u})$, $z_{u_\theta,u-\bar{u}} = \mathcal{S}'(u_\theta)(u-\bar{u})$, $z_{\bar{u},u-\bar{u}} = \mathcal{S}'(\bar{u})(u-\bar{u})$, and $w_{u_\theta} = \mathcal{D}(u_\theta)$. According to \Cref{J:secondder}, we get
 \begin{align*}
\begin{aligned}
	& \vert [\mathcal{J}''(u_\theta)-\mathcal{J}''(\bar{u})](u-\bar{u})^2\vert =  \int_Q\left\{ |z_{u_\theta,u-\bar{u}}|^2 -  |z_{\bar{u},u-\bar{u}}|^2 \right\}\du x\du t\\ 
	&+ 2\big(\big(z_{\bar{u},u - \bar{u}}\cdot \nabla\big) z_{\bar{u},u - \bar{u}}, w_{\bar{u}} \big)_Q -  2\big(\big(z_{{u}_{\theta},u - \bar{u}}\cdot \nabla\big) z_{{u}_{\theta},u - \bar{u}}, w_{{u}_{\theta}} \big)_Q.
\end{aligned}
\end{align*} 
Let us denote by $I_1$ and $I_2$ the integral corresponding to the square norms of the linear variables and trilinear form, respectively.

Estimate \eqref{estimate:linear-controlnstate1} aids us to get an upperbound for $I_1$
 \begin{align*}
\begin{aligned}
	|I_1| & = \left|\left( z_{u_\theta,u-\bar{u}} + z_{\bar{u},u-\bar{u}}, z_{u_\theta,u-\bar{u}} -  z_{\bar{u},u-\bar{u}} \right)_Q \right|\\
	& \le \left(\|z_{u_\theta,u-\bar{u}}\|_{L^2(Q)^2} + \|z_{\bar{u},u-\bar{u}} \|_{L^2(Q)^2} \right)
	\|z_{u_\theta,u-\bar{u}} -  z_{\bar{u},u-\bar{u}} \|_{L^2(Q)^2}\\
	& \le  c_1\left(\|z_{u_\theta,u-\bar{u}}\|_{L^2(Q)^2} + \|z_{\bar{u},u-\bar{u}} \|_{L^2(Q)^2} \right)\|y_{u_\theta} - y_{\bar{u}}\|_{C({\overline{Q}})^2}\| z_{\bar{u},u-\bar{u}}\|_{L^2(Q)^2}
\end{aligned}
\end{align*} 
Denoting by $\delta_1>0$ the constant mentioned in \Cref{lemma:lineargap-linear} we know that whenever $\|u_\theta-\bar{u} \|_{L^1(Q)^2}<\delta_1$ we get
 \begin{align*}
\begin{aligned}
	|I_1| & \le  c_1\|y_{u_\theta} - y_{\bar{u}}\|_{C({\overline{Q}})^2}\| z_{\bar{u},u-\bar{u}}\|_{L^2(Q)^2}^2.
\end{aligned}
\end{align*} 
Using for now the notations $z_{\bar{u}}:= z_{\bar{u},u-\bar{u}}$ and $z_{{u_\theta}}:=z_{{u_\theta},u-\bar{u}}$, estimates \eqref{estimate:adjdiff1} and \eqref{cc2:p_v} gives us an estimate for $I_2$,
 \begin{align*}
\begin{aligned}
	|I_2| =&\, \left| \big( (z_{\bar{u}}\cdot \nabla) z_{\bar{u}},w_{\bar{u}}\big)_Q - \big((z_{{u_\theta}}\cdot \nabla) z_{{u_\theta}}, w_{{u_\theta}} \big)_Q \right|\\
	\le&\, \left| \big(((z_{\bar{u}}-z_{{u_\theta}})\cdot \nabla) z_{\bar{u}}, w_{\bar{u}}\big)_Q \right| +  \left| \big((z_{{u_\theta}}\cdot \nabla)( z_{\bar{u}} - z_{{u_\theta}}), w_{\bar{u}}\big)_Q \right|+ \left| \big((z_{{u_\theta}}\cdot \nabla) z_{{u_\theta}}, w_{\bar{u}}- w_{{u_\theta}}\big)_Q \right|\\
	=&\, \left| \big(((z_{\bar{u}}-z_{{u_\theta}})\cdot \nabla) w_{\bar{u}}, z_{\bar{u}}\big)_Q \right| +  \left| \big((z_{{u_\theta}}\cdot \nabla)w_{\bar{u}}, z_{\bar{u}} - z_{{u_\theta}} \big)_Q \right| + \left| \big((z_{{u_\theta}}\cdot \nabla)(w_{\bar{u}}- w_{{u_\theta}}), z_{{u_\theta}} \big)_Q \right|\\
	\le &\, \|z_{\bar{u}} - z_{{u_\theta}}\|_{L^2(Q)^2}\|w_{\bar{u}}\|_{C(\overline I;C^1(\overline{\Omega})^2)}(\|z_{\bar{u}}\|_{L^2(Q)^2} + \|z_{{u_\theta}}\|_{L^2(Q)^2})\\
	& + \|z_{{u_\theta}}\|_{L^2(Q)^2}^2 \|w_{\bar{u}}- w_{{u_\theta}}\|_{C(\overline I;C^1(\overline{\Omega})^2)}
	\le  \left( 2c_1M_p +  \frac{9}{4}c\right)\|y_{\theta} - y_{\bar{u}}\|_{C({\overline{Q}})^2}\|z_{\bar{u},u-\bar{u}}\|_{L^2(Q)^2}^2.
\end{aligned}
\end{align*} 
We mention that to be able to reach the fifth relation we, once again, used the assumption that $\|u_\theta-\bar{u}\|_{L^1(\Omega)^2}<\delta_1$ where $\delta_1>0$ is as in Lemma \ref{lemma:lineargap-linear}.


Let $\tilde{s}\in [1,2)$ and $\tilde{s}' = \tilde{s}/(\tilde{s}-1)> 2$ be its H{\"o}lder conjugate, from \eqref{energy:stronglinearLP}, \eqref{state-to-controlgap} and \eqref{estimate:lsl1}, with the help of \eqref{uniformbound:control} we get
 \begin{align*}
&\vert [\mathcal J''(u_\theta)-\mathcal J''(\bar{u})](u-\bar{u})^2\vert  \le c\|y_\theta - y_{\bar{u}}\|_{C({\overline{Q}})^2}\|z_{\bar{u},u-\bar{u}}\|_{L^2(Q)^2}^2\\
& \le c\|u_\theta-\bar{u}\|_{L^{\tilde{s}'}(Q)^2}\|z_{\bar{u},u-\bar{u}}\|_{C(\overline{Q})^2}^{2-\tilde{s}} \|z_{\bar{u},u-\bar{u}}\|_{L^{\tilde{s}}(Q)^2}^{\tilde{s}} \le c\|u-\bar{u}\|_{L^{\tilde{s}'}(Q)^2}\|u-\bar{u}\|_{L^{\tilde{s}'}(Q)^2}^{2-\tilde{s}} \|u-\bar{u}\|_{L^{1}(Q)^2}^{\tilde{s}}\\
& \le cM_{\mathcal{U}}^{\frac{3-\tilde{s}}{\tilde{s}}}\|u-\bar{u}\|_{L^{1}(Q)^2}^{\frac{3-\tilde{s}}{\tilde{s}'} + \tilde{s}}.
\end{align*} 
By choosing $\tilde{s} > 3/2$, we can find $\ell_1,\ell_2>0 $ such that $2+\ell_1+\ell_2 = \frac{3-\tilde{s}}{\tilde{s}'} + \tilde{s}$. Thus, if we take $\mu = 1 + \ell_1$ and $\delta_2 = \left( \varepsilon/cM_{\mathcal{U}}^{\frac{3-\tilde{s}}{\tilde{s}}} \right)^{1/\ell_2}$ we get that whenever $\|u-\bar{u}\|_{L^{1}(Q)^2} < \delta:= \min\{\delta_1,\delta_2 \}$
 \begin{align*}
\vert [\mathcal J''(u_\theta)-\mathcal J''(\bar{u})](u-\bar{u})^2\vert & \le \varepsilon\|u-\bar{u}\|_{L^{1}(Q)^2}^{1+\mu}.
\end{align*} 
Furthermore, since  {$\tilde{s}\in (3/2,2)$} we have $\mu\in (1,2)$.  
\end{proof}


\section{Stability under perturbations}\label{sec:5}

This section deals with the main result of this paper. That is, we present the effects of perturbations on problem \eqref{opticon}.

\vspace{.1 in}\noindent{\bf  The perturbed problem.}
To begin, we consider the datum-perturbed Navier--Stokes equations 
 \begin{align}\label{dynsysper}
\left\{ \begin{aligned}
&\partial_t{y}- \nu\Delta {y}+ ({y}\cdot\nabla){y} + \nabla {p} = u\,\,\,\,\text{in Q},\\
&\dive {y}=0\,\,\,\,\text{in Q},\,\,  {y}=0\,\,\,\,\text{on $\Sigma$},\,\,\, {y}(\cdot,0)= y_0+\xi\,\,\,\text{in $\Omega$},
\end{aligned}
\right.
\end{align} 
where the perturbation $\xi\in W^{2-\frac{2}{\bar s}}_0(\Omega)^2$ accounts for possible uncertainty on the initial datum. From \Cref{existence:strongNS}, we know that the operator $\mathcal{S}_\xi: L^{\bar{s}}(Q) \to W_{\bar s}^{2,1}$ defined as $\mathcal{S}_\xi(u) = \widetilde{\mathcal{S}}(u,y_0+\xi)$ is well-defined, from which we know that the element $y_{u}^\xi :=\mathcal{S}_\xi(u) \in W_{\bar s}^{2,1}$ uniquely solves \eqref{dynsysper}.  Analogously with force-to-velocity map, $\mathcal{S}_\xi$ can be shown to be of class $C^\infty$, specifically, for given $u\in\mathcal{U}$ and $v\in L^{\bar{s}}(Q)^2$ with $u+v\in \mathcal{U}$ one gets an element $\mathcal{S}_\xi'(u)v\in W_{\bar s}^{2,1}$ that solves \eqref{strongwq:linear} with $\overline{y}_1=\overline{y}_2 = y_{u}^\xi $ and zero initial datum. Furthermore, we get the adjoint of $\mathcal{S}_\xi'(u) \in \mathcal{L}(L^{\bar{s}}(Q)^2;W^{2,1}_{\bar s,0})$, i.e., the element $\mathcal{S}_\xi'(u)^*v\in W^{2,1}_{\bar s} $ solves \eqref{weakop:adjoint} with $\overline{y}_1=\overline{y}_2 = y_u^\xi$ and ${v}$ as the right-hand side.

For $\eta\in L^{\bar s}(Q)$ and $\varepsilon\ge0$, we consider the perturbed objective function
 \begin{align*}
\mathcal J_{\xi,\eta}^\varepsilon(u):= \frac{1}{2}\int_{0}^T\int_\Omega |y_u^\xi(x,t) -  \big(y_{d}(x,t)+\eta(x,t)\big)|^2 \, dx
\, dt+\frac{\varepsilon}{2}\int_{0}^T\int_\Omega  |u|^2\, dx \, dt,
\end{align*} 
where $y_{u}^\xi =\mathcal{S}_\xi(u)$. 
The perturbation $\eta$ represent uncertainty in the tracking data and the parameter $\varepsilon$ is a weight parameter for the Tikhonov regularization term. 

The perturbed optimal control problem can now be written as 
 \begin{align}
\min_{u\in\mathcal{U}} \mathcal{J}_{\xi,\eta}^\varepsilon(u)\quad \text{ subject to }\eqref{dynsysper}.
\tag{P$_{\xi,\eta}^\varepsilon$}
\label{opticonper}
\end{align} 

 {Due to the convexity of the $L^2$-norm}, the functional $\mathcal J_{\xi,\eta}^\varepsilon$ is  {weakly} lower semicontinuous. Therefore, the existence  of at least one global minimizer of problem \eqref{opticonper} is guaranteed. Furthermore, one can obtain a first-order necessary condition for a minimizer $\widehat{u}\in\mathcal{U}$ of \eqref{opticonper} which we can write as $0\in \varepsilon\widehat{u} +w_{\widehat{u}}^\eta + N_{\mathcal{U}}(\widehat{u}) $ where $w_{\widehat{u}}^\eta := \mathcal{S}_\xi'(\widehat{u})^*(y_{\widehat{u}}^\xi - y_d - \eta)\in W_{\bar s}^{2,1}$, $y_{\widehat{u}}^\xi = \mathcal{S}_\xi(\widehat{u})$, and $N_{\mathcal{U}}(\widehat{u})$ is the normal cone to $\mathcal{U}$ at $\widehat{u}$, see $\eqref{normalcone}$. Let us introduce $\mathfrak{W}:= W^{2-\frac{2}{\bar s}}_0(\Omega)^2\times L^\infty(Q)^2 \times [0,+\infty)$ as the set of perturbations. We also say that an element $\widehat{u}\in \mathcal{U}$ is a solution to the first order necessary optimality condition of \eqref{opticonper} if it satisfies $0\in \varepsilon\widehat{u} +w_{\widehat{u}}^\eta + N_{\mathcal{U}}(\widehat{u}) $ where $w_{\widehat{u}}^\eta := \mathcal{S}_\xi'(\widehat{u})^*(y_{\widehat{u}}^\xi - y_d - \eta)\in W_{\bar s}^{2,1}$, $y_{\widehat{u}}^\xi = \mathcal{S}_\xi(\widehat{u})$. Note that the collection of such elements is non-empty, since any solution of \eqref{opticonper} is included in this collection.

Aside from the growth assumption imposed to obtain the second-order sufficient condition, we shall rely on a slightly modified growth assumption from which we can get the  {desired stability}. 
\begin{assumption}\label{growthasu}
Let $\mu\in[1,2)$ and $\bar u\in\mathcal U$.  Suppose that
 \begin{align*}
\mathcal J'(\bar u)(u-\bar u)+\mathcal J''(\bar u)(u-\bar u)^2\ge \|u-\bar u\|_{L^1(Q)^2}^{\mu+1}\quad \text{for all $u\in\mathcal U$ with $\|u-\bar u\|_{L^1(Q)^2}\le\delta$.}
\end{align*} 
\end{assumption}

 {Let us discuss the appearance of the first derivative in Assumption \ref{growthasu} and in \eqref{growthasunec}. For optimal control problems without constraints on the set of admissible controls, the first derivative vanishes for the optimal control and thus does not contribute to the growth. But in the case of control constraints, the first variation does not need to vanish and can even contribute to the growth. This holds especially when the optimal controls are of bang-bang structure. For instance, suppose the solution to the adjoint equation satisfies a structural condition, which is commonly assumed in the literature on bang-bang optimal control problems. Then, the first variation already fully contributes to the growth appearing in Assumption \ref{growthasu} and \eqref{growthasunec} and can even compensate for a negative second variation. }
We note that Assumption \ref{growthasu} is a weaker one compared to the growth assumption \eqref{growthasunec}, i.e., if $\bar u \in\mathcal{U}$ satisfies $0\in w_{\bar u} + N_{\mathcal{U}}(\bar{u})$ then Assumption \ref{growthasu} implies \eqref{growthasunec}. Both assumptions imply that $\bar{u}$ is bang-bang, and both will give us the stability we want. The advantage of assuming Assumption \ref{growthasu} is that we can directly apply the results in the Appendix while \eqref{growthasunec} requires a more nuanced proof to achieve stability. 

\begin{theorem}\label{Mainthm}
 {Let $\bar u$ be a local minimizer of problem \eqref{opticon}, $\mu\in[1,2)$ and  suppose that Assumption \ref{growthasu}
holds. Then there exist $\delta>0$ and $\kappa>0$ such that for all $\xi,\eta,\varepsilon\in\mathfrak{W}$ and all solutions $\widehat u\in\mathcal U$ of the first order necessary optimality condition of \eqref{opticonper} with $\|\widehat u-\bar u\|_{L^1(Q)^2}\le \delta$, it holds that
 \begin{align*}
\|\widehat u-\bar u\|_{L^1(Q)^2}\le \kappa\Big(\|\xi\|_{W^{2-\frac{2}{\bar s},s}_{0.\sigma}(\Omega)^2} +\|\eta\|_{L^2(Q)}+ \varepsilon\Big)^{\frac{1}{\mu}}.
\end{align*} }

\end{theorem}

 {It turns out that the weaker condition \eqref{growthasunec} is also sufficient for solution stability estimates. However, we need the controls corresponding to the perturbed problem to be global minimizers.}
\begin{theorem}\label{Mainthmw}
 {Let $\bar u$ be a local minimizer of problem \eqref{opticon}, $\mu\in[1,2)$ and suppose that \eqref{growthasunec}
holds. Then there exist $\delta>0$ and $\kappa>0$ such that for all $\xi,\eta,\varepsilon\in\mathfrak{W}$ and all global minimizers $\widehat u\in\mathcal U$ of problem \eqref{opticonper} with $\|\widehat u-\bar u\|_{L^1(Q)^2}\le \delta$, it holds that
 \begin{align*}
\|\widehat u-\bar u\|_{L^1(Q)^2}\le \kappa\Big(\|\xi\|_{W^{2-\frac{2}{\bar s},s}_{0.\sigma}(\Omega)^2} +\|\eta\|_{L^2(Q)}+ \varepsilon\Big)^{\frac{1}{\mu}}.
\end{align*} 
}
\end{theorem}

 {We postpone the proof of the theorems above as they require some estimates which are not yet available to our disposal. For now, we show that condition \eqref{growthasunec} is necessary to obtain the stability of the optimal controls under perturbations appearing in the normal cone. For this we apply the abstract result in the Appendix. To be able to apply Theorem \ref{Neccondcri}, we need to consider also linear control perturbations appearing in the objective functinonal. Thus let us define
 \begin{align*}
\mathcal J_{\xi,\eta}^{\varepsilon,\rho}(u):= \frac{1}{2}\int_{0}^T\int_\Omega |y_u^\xi(x,t) -  \big(y_{d}(x,t)+\eta(x,t)\big)|^2 \, dx
\, dt+\int_{0}^T\int_\Omega  u \rho\, dx \, dt+\frac{\varepsilon}{2}\int_{0}^T\int_\Omega  |u|^2\, dx \, dt,
\end{align*} 
and 
 \begin{align}
\min_{u\in\mathcal{U}} J_{\xi,\eta}^{\varepsilon,\rho}(u)\quad \text{ subject to }\eqref{dynsysper}.
\tag{P$_{\xi,\eta}^{\varepsilon,\rho}$}
\label{opticonperrho}
\end{align} }

\begin{theorem}
 {
Let $\bar u$ be a local minimizer of problem \eqref{opticon}, $\mu\in[1,2)$ and suppose that there exist $\delta>0$ and $\kappa>0$ such that for all $\xi,\eta,\varepsilon\in\mathfrak{W}$ and all solutions $\widehat u\in\mathcal U$ of the first order necessary optimality condition of \eqref{opticonper} with $\|\widehat u-\bar u\|_{L^1(Q)^2}\le \delta$, it holds that
 \begin{align*}
\|\widehat u-\bar u\|_{L^1(Q)^2}\le \kappa\Big(\|\xi\|_{W^{2-\frac{2}{\bar s},s}_{0.\sigma}(\Omega)^2} +\|\eta\|_{L^2(Q)}+ \varepsilon+\|\rho\|_{L^\infty(Q)}\Big)^{\frac{1}{\mu}}.
\end{align*} 
Then $\bar u$ satisfies \eqref{growthasunec} for appropriate constants $\delta$ and $c$.
}
\end{theorem}
\begin{proof}
 {The result is a direct consequence of Theorem \ref{Neccondcri}.}  
\end{proof}

To prove \Cref{Mainthm}, we will rely on the following lemmata that set up the application of the abstract results from the Appendix.
\begin{lemma}\label{lineargapper}
Let $u\in\mathcal U$ and $v\in L^{\bar s}(Q)^2$ such that $u+v\in \mathcal{U}$. Let $y_u^\xi = \mathcal{S}_\xi(u)$, $z_{u,v}^\xi := \mathcal{S}_\xi'(u)v\in W_{\bar s}^{2,1}$, $y_u = \mathcal{S}(u)$ and $z_{u,v} := \mathcal{S}'(u)v\in W_{\bar s}^{2,1}$. 
Then, there exists $c>0$ independent of $u$ such that
 \begin{align*}
\max\{\|z_{u,v}^\xi-z_{u,v}\|_{L^2(Q)^2} ,\|\nabla(z_{u,v}^\xi-z_{u,v})\|_{L^2(Q)^2}\}\le c\|y_u^\xi - y_u \|_{C(\overline{Q})^2}\|v\|_{L^2(Q)^2}.
\end{align*} 
\end{lemma}
\begin{proof}
We only show the inequality $\|z_{u,v}^\xi-z_{u,v}\|_{L^2(Q)^2}\le c\|y_u^\xi - y_u \|_{C(\overline{Q})^2}\|v\|_{L^2(Q)^2}$, as the other one follows analogously as in the forecoming arguments and the proof of \cite[Lemma 3.9]{CKbangbang}.

 {First, we notice that $B'(y_u - y_u^\xi)z_{u,v} \in L^{\bar s}(Q)^2$, because for any $\varphi \in  L^{\bar{s}'}(Q)^2$ we have
 \begin{align*}
\begin{aligned}
&\left| \langle B'(y_u - y_u^\xi)z_{u,v},\varphi\rangle_{L^2(V^*),L^2(V))} \right| = \left|  (((y_u-y_{\bar{u}})\cdot\nabla) z_{u,v}, \varphi)_Q + ((z_{u,v}\cdot\nabla)(y_u-y_{\bar{u}}),\varphi )_Q \right|\\
&\le  \left(\|y_u^\xi - y_u \|_{C(\overline{Q})^2}\|z_{u,v}\|_{W^{2,1}_{\bar{s}}} + \|z_{u,v}\|_{C(\overline{Q})^2}\|y_u^\xi - y_u \|_{W^{2,1}_{\bar{s}}} \right)\|\varphi\|_{L^{\bar{s}'}(Q)^2},
\end{aligned}
\end{align*}	 
where we can immediately see that all of the terms on the last line are bounded. Hence, we get that} the element $\mathfrak{Z} = z_{u,v}^\xi-z_{u,v} \in W^{2,1}_{\bar s}$ solves the linear system
 \begin{align*}
\begin{aligned}
\partial_t \mathfrak{Z} + \nu A\mathfrak{Z} + B'(y_u^\xi)\mathfrak{Z} + \nabla \mathfrak q&= B'(y_u - y_u^\xi)z_{u,v} &&\text{in }L^{\bar s}(Q)^2\\
\mathfrak{Z}(0)& = 0&&\text{in } W_0^{2-2/\bar s,\bar s}(\Omega)^2.
\end{aligned}
\end{align*} 
Now, let $f\in L^2(I;H)$ and consider the adjoint system
 \begin{align*}
\begin{aligned}
-\partial_t \mathfrak{w} + \nu A\mathfrak{w} + B'(y_u^\xi)^*\mathfrak{w} + \nabla \mathfrak{r}&= f &&\text{in }L^2(I;H)\\
\mathfrak{w}(T)& = 0&&\text{in } H.
\end{aligned}
\end{align*} 
From \Cref{theorem:weakadjoint}, we know that $\|\nabla\mathfrak w\|_{L^2(Q)^2} \le c\|f\|_{L^2(Q)^2}$ for some constant $c>0$.
Hence, we get that
 \begin{align*}
\begin{aligned}
&\int_Q \mathfrak{Z}\cdot f\du x\du t = \big(\mathfrak{Z}, -\partial_t \mathfrak{w} + \nu A\mathfrak{w} + B'(y_u^\xi)^*\mathfrak{w} + \nabla \mathfrak{r}\big)_Q \\
& = \big(\partial_t \mathfrak{Z} + \nu A\mathfrak{Z} + B'(y_u^\xi)\mathfrak{Z} + \nabla \mathfrak q,  \mathfrak w\big)_Q = \big\langle B'(y_u - y_u^\xi)z_{u,v}, \mathfrak w\rangle_{L^2(V^*),L^2(V)}\\
& \le c\|y_u - y_u^\xi\|_{C(\overline{Q})^2}\|z_{u,v}\|_{L^2(Q)^2}\|\nabla\mathfrak w\|_{L^2(Q)^2} \le c\|y_u - y_u^\xi\|_{C(\overline{Q})^2}\|z_{u,v}\|_{L^2(Q)^2}\|f\|_{L^2(Q)^2}.
\end{aligned}
\end{align*} 
The fact that $\|z_{u,v}\|_{L^2(Q)^2} \le c\|v \|_{L^2(Q)^2}$ proves the claim.

\end{proof}

\begin{lemma}\label{adjointgapper}
Let $u\in\mathcal U$. Let $y_u^\xi = \mathcal{S}_\xi(u)$, $w_u^\eta := \mathcal{S}_\xi'(u)^*(y_u^\xi - y_d - \eta)\in W_{\bar s}^{2,1}$.
Then, there exists $c>0$ (independent of $u$) such that
 \begin{align}\label{est:adjointgapper}
 {\|w_u^\eta-w_u\|_{W_{\bar s}^{2,1}}\le c\Big(\|\xi\|_{W^{2-\frac{2}{\bar s},\bar s}_{0.\sigma}(\Omega)^2}+\|\eta\|_{L^{\bar s}(Q)}\Big)}
\end{align} 
for all $\xi\in W^{2-\frac{2}{\bar s},\bar s}_{0.\sigma}(\Omega)^2$ and $\eta\in L^{\bar s}(Q)$, where $w_u = \mathcal{S}'(u)^*(y_u - y_d)$ and $y_u = {\mathcal{S}}(u)$.
\end{lemma}
\begin{proof}
We begin by obtaining some estimates for the element $z_u^\xi := y_u^\xi - y_u \in W^{2,1}_{\bar{s}} $ which solves the system
 \begin{align*}
\left\{
\begin{aligned}
\partial_tz_u^\xi + \nu Az_u^\xi + \widetilde{B}(y_u,y_u^\xi)z_u^\xi + \nabla q &= 0 &&\text{in }L^s(Q)^2,\\
z_u^\xi(0) & = \xi &&\text{in }W_{0,\sigma}^{2-2/\bar{s},\bar{s}}(\Omega)^2.
\end{aligned}\right.
\end{align*} 
According to \Cref{theorem:stronglinear}, this solution satisfies $\| z_u^\xi\|_{W^{2,1}_{\bar{s}}} \le c \|\xi\|_{W_{0,\sigma}^{2-2/\bar{s},\bar{s}}(\Omega)^2}$.

Now the difference $\mathfrak{w}_u^\eta:= w_u^\eta-w_u$ solves the system
 \begin{align*}
\begin{aligned}
-\partial_t \mathfrak{w}_u^\eta + \nu A \mathfrak{w}_u^\eta + B'(y_u^\xi)^*\mathfrak{w}_u^\eta + \nabla \mathfrak{r} & = B'(z_u^\xi)^*w_u + z_u^\xi + \eta &&\text{in }L^{\bar s}(Q)^2\\
\mathfrak{w}_u^\eta(T)& = 0 &&\text{in }W_0^{2-2/{\bar s},\bar s}(\Omega)^2.
\end{aligned}\label{system:adjointdiffpert}
\end{align*} 
Using the same arguments as in \Cref{adjointgap}, we get \eqref{est:adjointgapper}.  
\end{proof}

\begin{lemma}\label{Esslem}
Let $\bar u\in\mathcal U$ and $\mu\in[1,2)$ satisfy  {Assumption \ref{growthasu}}. There exist positive numbers $\delta$ and $c$ such that 
 \begin{align*}
\|u-\bar u\|_{L^1(Q)^2}\le c\|\rho\|_{L^\infty(Q)^2}^{\frac{1}{{\mu}}}
\end{align*} 
for all  $\rho\in L^\infty(Q)$  and $u\in\mathcal U$ satisfying $\rho\in w_u+ N_{\mathcal U}(u)$ and $\|u-\bar u\|_{L^1(Q)^2}\le\delta$.
\end{lemma}
\begin{proof}
By Proposition \ref{equivagro}, there exists positive constants $c$ and $\alpha$ such that $\mathcal{J}'(u)(u-\bar u)\ge c\|u-\bar u\|_{L^1(Q)^2}^{\mu+1}$ for all $u\in \mathcal U$ with $\|u-\bar u\|_{L^1(Q)^2}\le\alpha$. Given  $\rho\in L^\infty(Q)^2$  and $u\in\mathcal U$ satisfying $\rho\in w_u+ N_{\mathcal U}(u)$ and $\|u-\bar u\|_{L^1(Q)^2}\le\alpha$, it holds 
\begin{equation}\label{helpeq}
( w_u-\rho,\bar u-u)_Q \geq 0.
\end{equation}
Now rearranging the terms in \eqref{helpeq} and applying Proposition \ref{equivagro}, we find
 \begin{align*}
c\|u-\bar u\|_{L^1(Q)^2}^{\mu+1}\leq \mathcal{J}'(u)(u-\bar u)= ( w_u,u-\bar u)_Q \leq \|\rho\|_{L^\infty(Q)^2} \|\bar u-u\|_{L^1(Q)^2},
\end{align*} 
which proves the claim.  
\end{proof}

\begin{proof}[Proof of {\Cref{Mainthm}}]
Let us begin by defining $\rho:=w_{\widehat u} - \varepsilon \widehat u-w_{\widehat u}^\eta \in L^\infty(Q)^2$, where $w_{\widehat u} = \mathcal{S}'(\widehat{u})^*(y_{\widehat{u}}-y_d)$ and $w_{\widehat u}^\eta = \mathcal{S}_{\xi}'(\widehat{u})^*(y_{\widehat{u}}-y_d-\eta)$, and utilizing \Cref{adjointgapper} we get
 \begin{align*}
\|\rho\|_{L^\infty(Q)}\le \| w_{\widehat u}-w_{\widehat u}^\eta\|_{L^\infty(Q)^2} + \varepsilon\| \widehat u\|_{L^\infty(Q)^2}\le  c\Big(\|\xi\|_{W^{2-\frac{2}{\bar s},s}_{0.\sigma}(\Omega)^2}+\|\eta\|_{L^2(Q)}+\varepsilon\Big),
\end{align*} 
where $c>0$ is dependent on $M_{\mathcal{U}}$ and the constant in \Cref{adjointgapper}.
As $\widehat u\in\mathcal{U}$ is a local minimizer of problem \eqref{opticonper}, we get $0\in \varepsilon \widehat u+ w_{\widehat u}^\eta+ N_{\mathcal U}(\widehat u)$ and thus $\rho\in w_{\widehat{u}}+ N_{\mathcal U}(\widehat u)$. 
Then by \Cref{Esslem},  
 \begin{align*}
\begin{aligned}
\|\widehat u-\bar u\|_{L^1(Q)^2}\le c\|\rho\|_{L^\infty(Q)}^{\frac{1}{{\mu}}} &\le 
cc_2^{\frac{1}{{\mu}}}\Big(\|\xi\|_{W^{2-\frac{2}{\bar s},s}_{0.\sigma}(\Omega)^2}+\|\eta\|_{L^2(Q)}+\varepsilon\Big)^{\frac{1}{{\mu}}}\\
&:=\kappa\Big(\|\xi\|_{W^{2-\frac{2}{\bar s},s}_{0.\sigma}(\Omega)^2}+\|\eta\|_{L^2(Q)}+\varepsilon\Big)^{\frac{1}{\mu}}.
\end{aligned}
\end{align*} 
\end{proof}

\begin{proof}[Proof of {\Cref{Mainthmw}}]
We first note that since $\widehat u$ is a minimizer of Problem \eqref{opticonper} we have $J_{\xi,\eta}^\varepsilon(\widehat u) -\mathcal J_{\xi,\eta}^\varepsilon(\bar u)\leq 0$ which is equivalent to writing 
 \begin{align}\label{difference}
\begin{aligned}
0 & \le - ( \mathcal G(\widehat u)-\mathcal G(\bar u)) + (\eta, y_{\widehat u}^\xi- y_{\bar u}^\xi)_Q +\frac{\varepsilon}{2}\int_Q  |\bar u|^2-\vert \widehat u\vert ^2\, dx \, dt\\
& \le  - ( \mathcal G(\widehat u)-\mathcal G(\bar u)) + \|\eta\|_{L^2(Q)^2}\|y_{\widehat u}^\xi- y_{\bar u}^\xi\|_{L^2(Q)^2} + M_\mathcal{U}\varepsilon\|\bar u - \widehat{u}\|_{L^1(Q)^2}\\
& \le  - ( \mathcal G(\widehat u)-\mathcal G(\bar u)) + c(\|\eta\|_{L^{\bar s}(Q)^2} + \varepsilon)\|\bar u - \widehat{u}\|_{L^1(Q)^2}
\end{aligned}
\end{align} 
where for a fixed $\xi\in L^\infty(Q)^2$ we used the notation $\mathcal G(u):=\frac{1}{2}\int_Q |y_{u}^\xi(x,t) - y_{d}(x,t)|^2\, dx\,dt.$ We note that to achieve the last line we used the fact that $y_{\widehat u}^\xi- y_{\bar u}^\xi \in W^{2,1}_s$ solves an equation of the form \eqref{strongwq:linear} with the right-hand side $\widehat u - \bar u$ and initial datum equal to zero, which implies $\|y_{\widehat u}^\xi- y_{\bar u}^\xi\|_{L^2(Q)^2} \le c \|\widehat u - \bar u\|_{L^1(Q)^2}$ for some constant $c>0$.

By the Taylor theorem, for some $\theta\in[0,1]$ 
we find $u_{\theta} = \bar{u} + \theta(\widehat u-\bar u) \in \mathcal{U}$ such that 
\begin{equation*}
\mathcal G(\widehat u)-\mathcal G(\bar u):=\mathcal G'(\bar u)(\widehat u-\bar u)+\frac{1}{2}\mathcal G''( u_\theta)(\widehat u-\bar u)^2.
\end{equation*}
From this we rewrite \eqref{difference} as
 \begin{align*}
\mathcal J'(\bar u)(\widehat u-\bar u)+\frac{1}{2}\mathcal J''(u_\theta)(\widehat u-\bar u)^2 \le G_1 + \frac{1}{2}G_2 + c(\|\eta\|_{L^{\bar s}(Q)^2} + \varepsilon)\|\bar u - \widehat{u}\|_{L^1(Q)^2}
\end{align*} 
where $G_1 = \mathcal J'(\bar u)(\widehat u-\bar u) - \mathcal G'(\bar u)(\widehat u-\bar u)$ and $G_2 = \mathcal J''(u_\theta)(\widehat u-\bar u)^2 - \mathcal G''(u_\theta)(\widehat u-\bar u)^2$. Let us now get some estimates for $G_1$ and $G_2$, respectively.

For $G_1$, we note that we can write the derivative of $\mathcal{G}$ as $\mathcal{G}'(\bar{u})(\widehat u - \bar{u}) = (w_{\bar u}^{\eta}, \widehat{u}-\bar{u})_Q$ where $w_{\bar u}^{\eta} = \mathcal{S}_\xi'(\bar u)^*(y_{\bar u}^{\xi} - y_d - \eta)$ so that, by additionally employing \eqref{est:adjointgapper} and the embedding $W^{2,1}_{\bar s} \hookrightarrow C(\overline{Q})^2$, we get
 \begin{align}\label{g1}
|G_1| & \le \|w_{\bar u}^{\eta} - w_{\bar u} \|_{C(\overline{Q})^2}\|\widehat{u} - \bar{u} \|_{L^1(Q)^2} \le c\Big(\|\xi\|_{W^{2-\frac{2}{\bar s},\bar s}_{0,\sigma}(\Omega)^2}+\|\eta\|_{L^2(Q)}\Big)\|\widehat{u} - \bar{u} \|_{L^1(Q)^2}.
\end{align} 

The estimate for $G_2$ will be divided into two parts, i.e., using the same form as the second Fr{\'e}chet derivative of $\mathcal{J}$ from \Cref{derJ} we write $G_2 = G_{2,1} + G_{2,2}$ where
 \begin{align*}
\begin{aligned}
&G_{2,1} = \int_Q |z_{u_\theta,\widehat{u}-\bar{u}}|^2-|z_{u_\theta,\widehat{u}-\bar{u}}^\xi|^2 \du x\du t\\
&G_{2,2} = 2\int_Q [(z_{u_\theta,\widehat{u}-\bar{u}} \cdot \nabla)z_{u_\theta,\widehat{u}-\bar{u}}]\cdot w_{u_{\theta}} - [(z_{u_\theta,\widehat{u}-\bar{u}}^\xi \cdot \nabla)z_{u_\theta,\widehat{u}-\bar{u}}^\xi]\cdot w_{u_{\theta}}^\eta \du x\du t
\end{aligned}
\end{align*} 
Knowing that $\max\{ \|z_{u_\theta,\widehat{u}-\bar{u}}\|_{L^2(Q)^2},\|z_{u_\theta,\widehat{u}-\bar{u}}^\xi\|_{L^2(Q)^2} \} \le c\|\widehat{u}-\bar{u} \|_{L^2(Q)^2}$, and due to \Cref{lineargapper} we majorize $G_{2,1}$ as follows:
 \begin{align*}
\begin{aligned}
|G_{2,1}| &\le \|z_{u_\theta,\widehat{u}-\bar{u}}^\xi + z_{u_\theta,\widehat{u}-\bar{u}}\|_{L^2(Q)^2} \|z_{u_\theta,\widehat{u}-\bar{u}}^\xi - z_{u_\theta,\widehat{u}-\bar{u}}\|_{L^2(Q)^2}  \le c\|\widehat{u}-\bar{u} \|_{L^1(Q)^2}\|y^\xi_{u_\theta} - y_{u_\theta} \|_{C(\overline{Q})^2}.
\end{aligned}
\end{align*} 
Similarly, using additionally \Cref{adjointgapper}, we majorize the remaining term as
 \begin{align}\label{g21}
\begin{aligned}
|G_{2,2}| \le&\, \left| \big(((z_{u_\theta,\widehat{u}-\bar{u}}^\xi-z_{u_\theta,\widehat{u}-\bar{u}}) \cdot \nabla)z_{u_\theta,\widehat{u}-\bar{u}} , w_{u_{\theta}} \big)_Q \right| + \left|  \big((z_{u_\theta,\widehat{u}-\bar{u}}^\xi \cdot \nabla)(z_{u_\theta,\widehat{u}-\bar{u}}^\xi-z_{u_\theta,\widehat{u}-\bar{u}}), w_{u_{\theta}}\big)_Q \right|\\
& + \left| \big((z_{u_\theta,\widehat{u}-\bar{u}}^\xi \cdot \nabla)z_{u_\theta,\widehat{u}-\bar{u}}^\xi, w_{u_{\theta}}^\eta - w_{u_{\theta}}\big)_Q\right|\\
\le &\, \|z_{u_\theta,\widehat{u}-\bar{u}}^\xi-z_{u_\theta,\widehat{u}-\bar{u}} \|_{L^2(Q)^2}\|\nabla z_{u_\theta,\widehat{u}-\bar{u}}\|_{L^2(Q)^2}\|w_{u_\theta}\|_{C(\overline{Q})^2}\\
& + \|z_{u_\theta,\widehat{u}-\bar{u}}^\xi\|_{L^2(Q)^2}\|\nabla(z_{u_\theta,\widehat{u}-\bar{u}}^\xi-z_{u_\theta,\widehat{u}-\bar{u}})\|_{L^2(Q)^2}\|w_{u_\theta}\|_{C(\overline{Q})^2}\\
& + \|z_{u_\theta,\widehat{u}-\bar{u}}^\xi\|_{L^2(Q)^2}\|\nabla(z_{u_\theta,\widehat{u}-\bar{u}}^\xi-z_{u_\theta,\widehat{u}-\bar{u}})\|_{L^2(Q)^2} \|w_{u_\theta}^\eta-w_{u_\theta}\|_{C(\overline{Q})^2}\\
\le &\, c\|\widehat{u}-\bar{u} \|_{L^1(Q)^2}(\|y^\xi_{u_\theta} - y_{u_\theta} \|_{C(\overline{Q})^2} + \|\xi\|_{W^{2-\frac{2}{\bar s},\bar s}_{0,\sigma}(\Omega)^2}+\|\eta\|_{L^{\bar s}(Q)} )
\end{aligned}
\end{align} 
Following the proof in \Cref{adjointgapper}, we know that $\|y^\xi_{u_\theta} - y_{u_\theta} \|_{C(\overline{Q})^2} \le c\|\xi\|_{W^{2-\frac{2}{\bar s},s}_{0,\sigma}}$ for some constant $c>0$. 
With this relation, and the estimates \eqref{g1} and \eqref{g21} we get
 \begin{align}\label{amost}
\mathcal J'(\bar u)(\widehat u-\bar u)+\frac{1}{2}\mathcal J''(u_\theta)(\widehat u-\bar u)^2 \le  c(\|\xi\|_{W^{2-\frac{2}{\bar s},s}_{0,\sigma}(\Omega)^2} + \|\eta\|_{L^{\bar s}(Q)^2} + \varepsilon)\|\bar u - \widehat{u}\|_{L^1(Q)^2}.
\end{align} 

From \Cref{curveest}, for any $\varepsilon>0$ we can find a $\delta_1>0$ such that 
 \begin{align*}
\vert \mathcal J''(\bar u_\theta)(u-\bar u)^2-\mathcal J''(\bar u)(u-\bar u)^2\vert \le \rho \|u-\bar u\|_{L^{1}(Q)^2}^{\mu+1}
\end{align*} 
whenever $u\in\mathcal{U}$ satisfies $\|u-\bar u\|_{L^1(Q)^2}\le \delta_1$.
Combining this with \eqref{growthasunec} we estimate the left-hand side of the inequality above as
 \begin{align}\label{done}
\begin{aligned}
&\mathcal J'(\bar u)(\widehat u-\bar u)+\frac{1}{2}\mathcal J''(\bar u_\theta)(\widehat u-\bar u)^2\\
& = \mathcal J'(\bar u)(\widehat u-\bar u)+\frac{1}{2}\mathcal J''(\bar{u})(\widehat u-\bar u)^2 +\frac{1}{2}\left[ \mathcal J''(u_\theta)(\widehat u-\bar u)^2 - \mathcal J''(\bar{u})(\widehat u-\bar u)^2 \right]\\
&\ge \left(1-\frac{\varepsilon}{2}\right)\|u-\bar u\|_{L^1(Q)^2}^{\mu+1}
\end{aligned}
\end{align} 
whenever $\|u-\bar u\|_{L^1(Q)^2}\le \delta:=\min\{\delta_1,\delta_2 \}$, where $\delta_2>0$ is as in the assumption \eqref{growthasunec}. The arbitrariness of $\varepsilon>0$ allows us to choose $\varepsilon<2$. Combining \eqref{amost} and \eqref{done} proves our claim.  
\end{proof}



To finally end this section, we mention that according to \Cref{equivagro}, local minimizers $\bar u\in \mathcal{U}$ that satify \eqref{suffop} should also satisfy \eqref{growthasunec}. This implies that said local minimizers are bang-bang and should satisfy the stability we just proved. 

\section{Conclusion}

In this paper, we studied the well-known velocity tracking problem for the Navier--Stokes equations under the context of optimal control. As the objective functional intended for tracking the velocity contains no regularization for the control, the solution, which is assumed to satisfy a box constraint, can be expected to be of bang-bang type. The study of bang-bang optimal solution is the focus of this paper. Before delving into the main results, we analyzed the necessary and sufficient conditions of the optimization problem. The second-order sufficient condition was established under growth conditions, which as far as we are aware have never been used in the context of the Navier--Stokes equations which implies the bang-bang structure of the optimal controls.

The main results of this article are the sufficient and necessary conditions for the stability of optimal controls under several perturbations. In particular, the original optimization problem is perturbed in terms of the desired velocity, the initial data, and the objective functional itself via the Tikhonov regularization. The tools utilized for the study of the solution stability are owed to the so-called strong H{\"o}lder subregularity, which we discussed in the appendix. To be able to apply such tools for the Navier--Stokes equations, we used and, in some cases, improved existing stability estimates concerning the solutions of the nonlinear, the linearized, and the adjoint equations. A noteworthy addition is the $L^s-L^1$ stability for the Oseen equations and its adjoint enabled us to prove that the velocity tracking functional has a changing curvature of order $\mu \in [1,2)$. This result then allowed us to prove the desired solution stability of the optimal controls.

\appendix

\section{Appendix}
We now collect the stability results of the paper in an abstract framework, as the same principles can be applied to other types of optimization problems. We employ normed spaces since they constitute an adequate setting for our purposes; norms provide positively homogeneous measures for notions of growth and convergence, unlike general abstract metric spaces.

The results of this section focus mainly on necessary and sufficient conditions for stability of the first-order necessary conditions in optimization. For  convenience of the reader, this section is intended to be absolutely self-contained and independent of other sections.
{
Throughout the Appendix, unless otherwise stated,  $\big(U,\|\cdot\|_{U}\big)$ is a normed space and $\mathcal U$ a convex subset of $U$. We also consider a real-valued functional $\mathcal J:\mathcal U\to\mathbb R$. We will focus on stability properties associated to the minimization problem
\begin{align}\label{minprobabs}
\min_{u\in\mathcal U}\mathcal J(u).
\end{align}
In the context of optimal control, $\mathcal U$ is to be interpreted as the set of controls, and $\mathcal J:\mathcal U\to\mathbb R$ as the objective functional. We will see that the stability of the system of necessary conditions for problem (\ref{minprobabs}) is closely related to the growth conditions satisfied by $\mathcal J$ at a local minimizer.}
\subsection{A first-order variant of the Ekeland principle}

The first subsection is of technical nature and is devoted to recalling a few results of variational analysis that will be used later on. In particular, we state a first-order variant of the seminal Ekeland variational principle.

We begin  recalling the standard notion of (first-order) Gateaux differentiability.  We say that  $\mathcal J$ is Gateaux differentiable at $\bar u\in\mathcal U$ if there exists a linear mapping $\mathcal J'(u)\in U^*$ such that 
 \begin{align*}
{\mathcal J}'(\bar u)v=\lim_{\varepsilon\to0^+}\frac{{\mathcal J}(\bar u+\varepsilon v)-{\mathcal J}(\bar u)}{\varepsilon}\quad \forall v\in \mathcal U-u.
\end{align*}


Working with functions defined on convex domains has the advantage of simple tangent and normal cone formulations, which in turn implies that the first-order necessary condition also take a simpler form.

The  normal cone to $\mathcal U$ at $\bar u$ is defined by
 \begin{align}\label{normalcone}
N_{\mathcal U}(\bar u):=\left\lbrace \rho\in U^*:\text{ }\rho(u-\bar u)\le0\quad\text{for all $u\in\mathcal U$}\right\rbrace.
\end{align}

The first-order necessary condition is well-known for Gateaux differentiable functions,  {and a lot of the work carried out in optimization and variational analysis relies on it}; see \cite[pp. 11-13]{LecVar}.  {If  ${\mathcal J}$ is G{a}teaux differentiable at a local minimizer $\bar u\in\mathcal U$, then $0\in \mathcal J'(\bar u)+N_{\mathcal U}(\bar u)$}.
Second-order necessary conditions for optimality are also well known; see, e.g., \cite[Lemma 3.44]{PerAna} or \cite[Theorem 3.45]{Andrz}.

We give now a technical lemma based on 
the celebrated Sion Minimax Theorem.
\begin{lemma}\label{normlem}
Let $\psi:\mathcal{U}\to\mathbb R$ be a convex lower semicontinuous function. Let $\hat u\in U$ and $\gamma>0$.  There exists $\hat\rho\in U^*$ with  $\|\hat \rho\|_{U^*}\le\gamma$ such that   
 \begin{align*}
\inf_{u\in M} \Big\{\psi(u) + \gamma\|u-\hat u\|_U\Big\} =	\inf_{u\in \mathcal{U}}\left\lbrace \psi(u)-\hat\rho (u-\hat u)\right\rbrace .
\end{align*} 
\end{lemma}
\begin{proof}
Let $\mathbb B^*$ be the unit ball of $U^*$. Define $f:\mathcal{U}\times \mathbb{B}^*\to\mathbb R$ by $f(u,\rho):=\psi(u)+\gamma\rho(u-\hat u)$. Note that $U^*$ endowed with the weak* topology is a linear topological space, and $\mathbb B^*$ is weak* compact by Banach-Alaoglu Theorem. The function  $f(\cdot,\rho):\mathcal{U}\to\mathbb R$ is convex and lower semicontinuous for each $\rho\in U^*$. The function $f(u,\cdot): \mathbb{B}^*\to\mathbb R$ is weak* continuous and affine for each $u\in \mathcal{U}$. The hypotheses of Sion Minimax Theorem (\cite[Corollary 3.3]{Sion}) are then satisfied, and hence 
 \begin{align}\label{InfSup}
\inf_{u\in \mathcal{U}}\sup_{\rho\in \mathbb B^*}\left\lbrace  \psi(u)+\gamma\rho(u-\hat u)\right\rbrace =\sup_{\rho\in \mathbb B^*}\inf_{u\in \mathcal{U}} \{\psi(u)+\gamma\rho(u-\hat u)\}.
\end{align} 
Let $h:\mathbb B^*\to \mathbb R$ be given by $h(\rho):=\inf_{u\in \mathcal{U}}\{\psi(u)+\gamma\rho(u-\hat u)\}$. Clearly, $h$ is weak* upper semicontinuous as it is the infimum of weak* continuous functions; and since $\mathbb B^*$ is  weak* compact, there exists $\rho^*\in\mathbb B^*$ such that $\sup_{\rho\in\mathbb B^*} h(\rho)=h(\rho^*)$.  This implies 
 \begin{align}\label{infsu}
\sup_{\rho\in\mathbb B^*}\inf_{u\in \mathcal{U}}\{\psi(u)+\gamma\rho(u-\hat u)\}=\inf_{u\in \mathcal{U}}\{\psi(u)+\gamma\rho^*(u-\hat u)\}.
\end{align} 
Finally, by (\ref{InfSup}) and (\ref{infsu}),
 \begin{align*}
\inf_{u\in \mathcal{U}} \Big\{\psi(u) + \gamma\|u-\hat u\|_U\Big\}&=\inf_{u\in \mathcal{U}}\sup_{\rho\in \mathbb B^*}\Big\{\psi(u) +\gamma\rho(u-\hat u)\Big\}=\inf_{u\in \mathcal{U}}\Big\{\psi(u) +\gamma\rho^*(u-\hat u)\Big\}.
\end{align*} 
The results follow defining $\hat \rho:=-\gamma\rho^*$.  
\end{proof}


We can now prove the following variant of Ekeland principle.
\begin{lemma}\label{EkelandLinNor}
Suppose $\big(U,\|\cdot\|_{U}\big)$ is a Banach space,  $\mathcal U$ is a closed convex subset of $U$, and  ${\mathcal J}:\mathcal U\to\mathbb R$ is a lower semicontinuous G{a}teaux differentiable function. Let $\bar u\in U$ and $r>0$ such that
 \begin{align*}
{\mathcal J}(\bar u)\le{\mathcal J}(s)\quad \text{for all $s\in\mathcal U$ with $\|s-\bar u\|_{U}\le r$}.
\end{align*} 
Let  $u\in \mathcal U$ and $\varepsilon>0$ satisfy $
\|u-\bar u\|_{U}<r$ and  ${\mathcal J}(u)\le  {\mathcal J}(\bar u)+\varepsilon.$
Then for every $\lambda\in\big(0,r-\|u-\bar u\|_U\big)$ there exist $\hat u\in\mathcal U$ and $\hat \rho\in U^*$ such that 
\begin{itemize}
\item[(i)] $\|u-\hat u\|_{U}\le \lambda$;
\item[(ii)] $\displaystyle\|\hat\rho\|_{U^*}\le \frac{\varepsilon}{\lambda}$;
\item[(iii)] $\hat\rho\in \mathcal J'(\hat u)+N_{\mathcal U}(\hat u)$.
\end{itemize}
\end{lemma}

\begin{proof}
Let $S:=\{s\in\mathcal U: \|s-\bar u\|_{U}\le r\}$. Since $S$ is a closed subset of $U$,  $S$ is a complete metric space endowed with the metric induced from the norm of $U$, and ${\mathcal J}|_S$ is lower semicontinuous. We can apply Ekeland Principle (\cite[Theorem 1.1]{Ekeland}) to obtain that for every $\lambda\in\big(0,r-\|u-\bar u\|_U\big)$ there exists $\hat u\in S$ such that 

\begin{itemize}
\item[(a)] $\|u-\hat u\|_{U}\le \displaystyle\lambda$;
\item[(b)] ${\mathcal J}(\hat u)\le {\mathcal J}(u)$;
\item[(c)] $\hat{\mathcal J}(\hat u)\le \hat{\mathcal J}(s)$ for all $s\in S$, where  $\hat {\mathcal J}:\mathcal U\to\mathbb R$ is given by 
 \begin{align*}
\hat{\mathcal J}(s):={\mathcal J}(s)+\frac{\varepsilon}{\lambda}\|s-\hat u\|_{U}.
\end{align*} 
\end{itemize}

Let $r_{\lambda}:=r-\|u-\bar u\|_{U}-\lambda$; clearly $r_{\lambda}>0$. If $s\in\mathcal U$ satisfies $\|s-\hat u\|_{U}\le r_{\lambda}$, then 
 \begin{align*}
\|s-\bar u\|_{U}\le 	\|s-\hat u\|_{U}+	\|\hat u-u\|_{U}+	\|u-\bar u\|_{U}\le r_{\lambda}+\lambda+\|u-\bar u\|_{U}=r.
\end{align*} 	
Thus, from item $(c)$, $\hat{\mathcal J}(\hat u)\le \hat{\mathcal J}(s)$ for all $s\in \mathcal U$ with  $\|s-\hat u\|_{U}\le r_{\lambda}$. We conclude that  $\hat u$ is a local minimizer of $\hat{\mathcal J}$.
From this, we get
 \begin{align*}
0\le \liminf_{t\to0^+}\hspace*{0.05cm}\frac{\hat{\mathcal J}(\hat u+t(s-\hat u))-\hat {\mathcal J}(\hat u)}{t}=\mathcal J'(\hat u)(s-\hat u)+\frac{\varepsilon}{\lambda}\|s-\hat u\|_{U}\quad\forall s\in\mathcal U.
\end{align*} 
This can be rewritten as 
 \begin{align*}
0=\inf_{s\in\mathcal U}\Big\{\mathcal J'(\hat u)(s-\hat u)+\frac{\varepsilon}{\lambda}\|s-\hat u\|_{U}\Big\}
\end{align*} 
By Lemma \ref{normlem}, there exists $\hat \rho\in U^*$ with $\|\hat\rho\|_{U^*}\le \varepsilon/\lambda$ such that
 \begin{align*}
0=\inf_{s\in\mathcal U}\Big\{\mathcal J'(\hat u)(s-\hat u)+\frac{\varepsilon}{\lambda}\|s-\hat u\|_{U}\Big\}=\inf_{s\in\mathcal U}\Big\{\mathcal J'(\hat u)(s-\hat u)-\hat\rho(s-\hat u)\Big\}\le \mathcal J'(\hat u)(v-\hat u)-\hat\rho(v-\hat u)
\end{align*} 
for all $v\in\mathcal U$. This implies $\hat\rho\in \mathcal J'(\hat u)+N_{\mathcal U}(\hat u)$. Clearly, $\hat u$ and $\hat\rho$ satisfy items $(i)$-$(iii)$.  
\end{proof}

\subsection{Strong H\"older subregularity of the optimality mapping}

This subsection is devoted to study the behavior of critical points under the presence of perturbations.
We derive necessary and sufficient conditions for stability of the variational inequality describing the first-order necessary condition at critical points. From this point on, we assume that $\mathcal{J}:\mathcal{U}\to \mathbb{R}$ is Gateaux differentiable, unless we specify otherwise.

\vspace{.1in}\noindent{\bf Stability of the first-order necessary conditions.}
In the literature, the stability of the first-order necessary conditions is studied as a property of a set-valued mapping encapsulating the (generalized) equation satisfied by local minimizers. This property is known as strong H\"older (metric) subregularity, see \cite[ Section 3I]{DontRock}; the property has also appeared in the literature by the name of strong (metric) $\theta$-subregularity, see \cite[Section 4]{Dontsubreg}.

Let us begin giving a suitable notion for the correspondence between solutions of the perturbed variational inequality and the perturbations.
The set-valued mapping $\Phi:\mathcal U \twoheadrightarrow U^*$ given by 
 \begin{align*}
\Phi(u):=\mathcal J'(u)+N_{\mathcal U}(u)
\end{align*} 
is called  \textit{the optimality mapping}.
We now give the definition of stability that we wish to analyze, i.e., the so-called \textit{strong (metric) subregularty}.
\begin{definition}
Let $\bar u\in\mathcal U$ satisfy $0\in\Phi(\bar u)$. We say that the optimality mapping $\Phi:\mathcal U\twoheadrightarrow U^*$ is strongly (H\"older) subregular at $\bar u$ (with exponent $\theta\in(0,\infty)$) if there exist positive numbers $\alpha$ and $\kappa$ such that the following property holds. For all $u\in\mathcal U$ and $\rho\in U^*$, 
 \begin{align}\label{Stapropec}
\|u-\bar u\|_U\le\alpha\quad\text{and}\quad\rho\in \mathcal J' (u)+N_{\mathcal U}(u)\quad\text{imply}\quad\|u-\bar u\|_{U}\le \kappa \|\rho\|_{U^*}^{\theta}.
\end{align} 
\end{definition}
We now proceed to state both sufficient and necessary conditions for this notion of stability.

\vspace{.1in}\noindent{\bf Sufficient conditions.}
The proof of the sufficient condition for stability, as shown in the next theorem, follows the arguments presented in \cite[Theorem 1]{AlSoz} to the letter, where it was previously proven in the context of optimal control.
\begin{theorem}\label{Mainprop}
Let $\bar u\in\mathcal U$ such that $0\in\Phi(\bar u)$, and $\mu\in(0,\infty)$. Suppose there exist positive numbers $\delta$ and $c$ such that
 \begin{align}\label{Grothwk0}
\mathcal J'(u)(u-\bar u)\ge c\|u-\bar u\|_{U}^{\mu+1}\quad\text{for all $u\in \mathcal U$ with $\|u-\bar u\|_{U}\le\delta$.}
\end{align} 
Then the optimality mapping $\Phi:\mathcal U\twoheadrightarrow U^*$ is strongly H\"older subregular at $\bar u$ with exponent $1/\mu$.
\end{theorem}
\begin{proof}
Let $u\in\mathcal U$ and $\rho\in U^*$ be arbitrary satisfying $\|u-\bar u\|_{U}\le\delta$ and $\rho\in\Phi(u)$. Then, as $\rho-\mathcal J'(u)\in N_{\mathcal U}(u)$ and $\bar u\in\mathcal U$, we have
 \begin{align*}
0\ge \big(\rho-\mathcal J'(u)\big)(\bar u-u)=\rho(\bar u-u)+\mathcal J'(u)(u-\bar u)\ge -\|\rho\|_{U^*}\|u-\bar u\|_{U} + c\|u-\bar u\|_{U}^{\mu+1}.
\end{align*} 
Hence, $\|u-\bar u\|_{U}\le c^{-1/\mu}\|\rho\|_{U^*}^{1/\mu}.$ The result follows defining $\alpha:=\delta$ and $\kappa:=c^{-1/\mu}$.  
\end{proof}

Growth assumption (\ref{Grothwk0}) appeared first in \cite[Assumption 2]{AlSoz} as a natural hypothesis for an affine optimal control problem; see also \cite[Proposition 4.3]{Alellip}, where this kind of growth was proven for an elliptic optimal control problem under a linearized growth hypothesis. A similar assumption of this type appeared in \cite[Assumption A2]{Alpafa}, where stability results for an affine optimal control problem were studied. In  Proposition \ref{Dergrowth} below, we give  further details on  growth (\ref{Grothwk0}) and its linearization.

\vspace{.1in}\noindent{\bf Necessary conditions.}
In order to establish necessary conditions for stability in the form of growth properties of functionals, we will use Ekeland principle in the form of Lemma \ref{EkelandLinNor}, following the approach used in \cite{AragHil,AragBan}. In those papers, the subregularity property of the subdifferential of convex functions was characterized in terms of quadratic growth conditions; see also \cite{Mordusub}, where a similar approach was used for the limiting subdifferential. In all those three papers only Lipschitz stability and quadratic growth conditions were considered. We make simple refinements in those arguments to consider both H\"older stability and higher-order growth conditions.

In the next theorem, we argue similarly to the proof of \cite[Theorem 3.3]{AragHil}; see also the proofs of \cite[Theorem 2.1]{AragBan} and \cite[Theorem 3.1]{Mordusub} for parallel arguments.

\begin{theorem}\label{Neccondcri}
Suppose $\big(U,\|\cdot\|_{U}\big)$ is a Banach space, $\mathcal U$ a closed convex subset of $U$ and ${\mathcal J}:\mathcal U\to\mathbb R$ a lower semicontinuous G{a}teaux differentiable function. Let $\bar u\in\mathcal U$ be a local minimizer of ${\mathcal J}$ and $\mu\in(0,\infty)$. Suppose that the optimality mapping $\Phi:\mathcal U\twoheadrightarrow U^*$ is strongly H\"older subregular at $\bar u$ with exponent $1/\mu$. Then there exist positive numbers $\delta$ and $c$ such that
 \begin{align}\label{growthmu}
{\mathcal J}(u)-{\mathcal J}(\bar u)\ge c\|u-\bar u\|_{U}^{\mu+1}\quad \text{for all $u\in \mathcal U$ with $\|u-\bar u\|_{U}\le\delta$.}
\end{align} 

\end{theorem}
\begin{proof}
Let $\alpha$ and $\kappa$ be positive numbers such that property (\ref{Stapropec}) holds. Suppose that (\ref{growthmu}) does not hold. Then there would exist $u\in\mathcal U\setminus\{\bar u\}$ satisfying $\|u-\bar u\|_{U}<2\alpha/3$ such that 
 \begin{align}\label{Contgrowth}
{\mathcal J}(u)-{\mathcal J}(\bar u)<\frac{1}{2^{2\mu+1}\kappa^{\mu}}\|u-\bar u\|_{U}^{\mu+1}.
\end{align} 
Let $\varepsilon:=2^{-(2\mu+1)}\kappa^{-\mu}\|u-\bar u\|_{U}^{\mu+1}$ and $\lambda:=2^{-1}\|u-\bar u\|_{U}$. Note that
 \begin{align*}
\lambda=\frac{3}{2}\|u-\bar u\|_{U}-\|u-\bar u\|_{U}<\alpha-\|u-\bar u\|_{U}.
\end{align*} 
From Lemma \ref{EkelandLinNor}, we conclude the existence of $\hat u\in\mathcal U$ and $\hat\rho\in U^*$ such that 
\begin{itemize}
\item[(i)] $\displaystyle\|u-\hat u\|_{U}\le \frac{1}{2}\|u-\bar u\|_{U}$;
\item[(ii)] $\displaystyle\|\hat\rho\|_{U^*}\le\frac{1}{4^{\mu}\kappa^\mu} \|u-\bar u\|_{U}^{\mu}$;
\item[(iii)] $\hat \rho\in \mathcal{J}'(\hat u)+N_{\mathcal U}(\hat u)$.
\end{itemize}
Observe that $\|\hat u-\bar u\|_{U}\le 	\|\hat u-u\|_{U}+	\|u-\bar u\|_{U}\le 1/2\|u-\bar u\|_{U}+\|u-\bar u\|_{U}< \alpha$.
By subregularity of the optimality mapping at $\bar u$, we get
 \begin{align}\label{in12}
\|\hat u-\bar u\|_{U}\le\kappa\|\hat\rho\|_{U^*}^{1/\mu}\le\frac{1}{4}\|u-\bar u\|_{U}.
\end{align} 
By item $(i)$, we have $\|u-\bar u\|_{U}\le \|u-\hat u\|_{U}+\|\hat u-\bar u\|_{U}\le 1/2\|u-\bar u\|_{U}+\|\hat u-\bar u\|_{U}$.
This implies $\|u-\bar u\|_{U}\le2\|\hat u-\bar u\|_{U}$. Combining this with (\ref{in12}), we get 
 \begin{align*}
\|\hat u-\bar u\|_{U}\le\frac{1}{4}\|u-\bar u\|_{U}\le\frac{1}{2}	\|\hat u-\bar u\|_{U},
\end{align*} 
and hence $\hat u=u=\bar u$. A contradiction to (\ref{Contgrowth}).  
\end{proof}

Growth condition (\ref{growthmu}) is well known in optimization. See, for example, \cite[Theorem 2.4]{Cawasem} or \cite[Theorem III]{IdrissParabolic} in affine optimal control; and \cite[Theorem 1]{IdrissEig} or \cite[Theorem I]{IdrissBalet} in the quantitative study of eigenvalues stability for the Schr\"{o}dinger operator.

\subsection{H\"older growth of real-valued functions}
In this section, we study how to reduce growth conditions (\ref{Grothwk0}) and (\ref{growthmu}) to linearized versions. This is to facilitate the understanding of their feasibility. We say that $\mathcal J:\mathcal U\to\mathbb R$ has second variation at $\bar u\in\mathcal U$ if there exists a  function $\mathcal J''(\bar u):U\times U\to\mathbb R$, positively homogeneous in each variable, such that
 \begin{align*}
{\mathcal J''}(\bar u)(v,w)=\lim_{\varepsilon\to0^+}\frac{{\mathcal J'}(\bar u+\varepsilon v)w-{\mathcal J'}(\bar u)w}{\varepsilon} \quad \forall v,w\in\mathcal U-\bar u.
\end{align*} 
From now on, we will assume that $\mathcal{J}$ has second variation at every element of $\mathcal U$. We abbreviate $\mathcal J''(\bar u)v^2:=\mathcal J''(\bar u)(v,v)$.

\vspace{.1in}\noindent{\bf A H\"older-type second order condition.}
In order to transfer  conditions (\ref{Grothwk0}) and (\ref{growthmu}) from being satisfied by a nonlinear function to a second order polynomial,  we will employ the following weakened version of ``twice continuously differentiable".
\begin{definition}
Let $\bar u\in\mathcal U$ and  $\mu\in[1,\infty)$. We say that ${\mathcal J}$ has changing curvature of order $\mu$ at $\bar u$ if for every $\varepsilon>0$ there exists $\delta>0$ such that 
 \begin{align}\label{thetacurv}
|\mathcal J''(\bar u+v)v^2-\mathcal J''(\bar u)v^2|\le \varepsilon\|v\|_{U}^{\mu+1}
\end{align} 
for all $v\in U$ with $\bar u+v\in\mathcal U$ and $\|v\|_{U}\le\delta$.
\end{definition}
Properties like (\ref{thetacurv}) have appeared ubiquitously in the optimal control literature of bang-bang controls. See, for example,  \cite[p. 4207]{Cawasem}, where it appeared as a standard assumption in abstract optimal control; or \cite[Lemma 11]{corella2023} where it appeared as natural property in the context of  parabolic optimal control problems.

\begin{proposition}\label{Propshila}
Suppose that ${\mathcal J}$ has changing curvature of order $\mu\in[1,\infty)$ at $\bar u\in\mathcal U$, then the following statements hold:
\begin{itemize}

\item[(i)]	For all $\varepsilon>0$ there exists $\delta>0$ such that 
 \begin{align*}
|{\mathcal J}(\bar u+v)-{\mathcal J}(\bar u)-\mathcal J'(\bar u)v-\frac{1}{2}\mathcal J''(\bar u)v^2|\le\varepsilon\|v\|_{U}^{\mu+1} 
\end{align*} 
for all $v\in U$ satisfying $\bar u+v\in\mathcal U$ and $\|v\|_U\le\delta$.

\item[(ii)]		 For all $\varepsilon>0$ there exists $\delta>0$ such that 
 \begin{align*}
|\mathcal J'(\bar u+v)v-\mathcal J'(\bar u)v-\mathcal J''(\bar u)v^2|\le\varepsilon\|v\|_{U}^{\mu+1} 
\end{align*} 
for all $v\in U$ satisfying $\bar u+v\in\mathcal U$ and $\|v\|_U\le\delta$.

\end{itemize}

\end{proposition}

\begin{proof}
For each $v\in U$ with $\bar u+v\in\mathcal U$, define $h_v:[0,1]\to\mathbb R$ by $h_v(t)={\mathcal J}(\bar u+tv)$. We can apply the Taylor Theorem to conclude that for each $v\in\mathcal U-\bar u$ there exists  $t_v\in[0,1]$ such that $h(1)-h(0)=h'(0)+h''(t_v)/2$. That is,
 \begin{align}\label{geninex}
{\mathcal J}(\bar u+v)-{\mathcal J}(\bar u)=\mathcal J'(\bar u)v+\frac{1}{2}\mathcal J''(\bar u+t_v v)v^2\quad\forall v\in\mathcal U-\bar u.
\end{align} 
Let $\varepsilon>0$ be given. By definition of changing curvature of order $\mu$ at a point,  we can find $\delta>0$ such that 
 \begin{align}\label{geniney}
|\mathcal J''(\bar u+t_v v)(t_v v)^2-\mathcal J''(\bar u)(t_v v)^2|\le 2\, \varepsilon\|t_v v\|_{U}^{\mu+1}
\end{align} 
whenever $v\in\mathcal U-\bar u$ satisfies $\|v\|_{U}\le\delta$. Then, combining (\ref{geninex}) and (\ref{geniney}), we get
 \begin{align*}
t_{v}^{2}|{\mathcal J}(\bar u+v)-{\mathcal J}(\bar u)-\mathcal J'(\bar u)v-\frac{1}{2}\mathcal J''(\bar u)v^2|&=\frac{1}{2}|\mathcal J''(\bar u+t_v v)v^2-\mathcal J''(\bar u)v^2|\le  t_v^{\mu+1}\varepsilon\|v\|_{U}^{\mu+1} 
\end{align*} 
for all $v\in U$ satisfying $\bar u+v\in\mathcal U$ and $\|v\|_U\le\delta$.  Since $\mu\ge 1$, it follows that 
 \begin{align*}
|{\mathcal J}(\bar u+v)-{\mathcal J}(\bar u)-\mathcal J'(\bar u)v-\frac{1}{2}\mathcal J''(\bar u)v^2|\le  t_{v}^{\mu-1}\varepsilon\|v\|_{U}^{\mu+1}\le \varepsilon\|v\|_{U}^{\mu+1}  
\end{align*} 
for all $v\in U$ satisfying $\bar u+v\in\mathcal U$ and $\|v\|_U\le\delta$. Thus, item $(i)$ holds.

The proof of item $(ii)$ is analogous; it follows defining $k_v:[0,1]\to \mathbb R$ given by $k_v(t):=\mathcal{J}'(\bar u+tv)v$ for each $v\in\mathcal U-\bar u$, and applying the Mean Value Theorem to each function $k_v$.  
\end{proof}

\vspace{.1 in}\noindent{\bf Growth of functionals and their differentials.}
One easy consequence of Proposition \ref{Propshila} is the following characterization of growth condition (\ref{growthmu}) which follows directly from item $(i)$ of Proposition \ref{Propshila}.

\begin{proposition}\label{equivagro}
Suppose that ${\mathcal J}$ has changing curvature of order $\mu\in[1,\infty)$ at $\bar u\in\mathcal U$, then the following statements are equivalent:

\begin{itemize}

\item[(i)]	There exist positive numbers $\alpha$ and $c$ such that
 \begin{align*}
{\mathcal J}(u)-{\mathcal J}(\bar u)\ge c\|u-\bar u\|_{U}^{\mu+1}\quad \text{for all $u\in \mathcal U$ with $\|u-\bar u\|_{U}\le\alpha$.}
\end{align*} 

\item[(ii)]	There exist positive numbers $\alpha$ and $c$ such that
 \begin{align*}
\mathcal J'(\bar u)(u-\bar u)+\frac{1}{2}\mathcal J''(\bar u)(u-\bar u)^2\ge c\|u-\bar u\|_{U}^{\mu+1}\quad \text{for all $u\in \mathcal U$ with $\|u-\bar u\|_{U}\le\alpha$.}
\end{align*} 
\end{itemize}
\end{proposition}
 {The growth in  previous proposition has been proved in the literature of optimal control several times, this usually involves using a growth condition on the second variation over a critical cone and the so-called structural assumption; this is a condition one the level sets of the adjoint variable. See \cite[Section 3]{casas2012} and \cite[Theorem 2.4]{Casasbang}.}

On the other hand, another trivial, but important, consequence of Proposition \ref{Propshila} is the characterization of growth condition (\ref{Grothwk0}).	
\begin{proposition}\label{Dergrowth}
If ${\mathcal J}$ has changing curvature of order $\mu\in[1,\infty)$ at $\bar u\in\mathcal U$, then the following statements are equivalent.

\begin{itemize}

\item[(i)]	There exist positive numbers $\alpha$ and $c$ such that
 \begin{align*}
\mathcal J'(u)(u-\bar u)\ge c\|u-\bar u\|_{U}^{\mu+1}\quad \text{for all $u\in \mathcal U$ with $\|u-\bar u\|_{U}\le\alpha$.}
\end{align*} 

\item[(ii)]	There exist positive numbers $\alpha$ and $c$ such that
 \begin{align*}
\mathcal J'(\bar u)v+\mathcal J''(\bar u)v^2\ge c\|u-\bar u\|_{U}^{\mu+1}\quad \text{for all $u\in \mathcal U$ with $\|u-\bar u\|_{U}\le\alpha$.}
\end{align*} 
\end{itemize}
\end{proposition}
 {We mention that this characterization has appeared before in PDE-constrained optimization; see \cite[Proposition 4.1]{Alellip} or \cite[Lemma 12]{corella2023}. Numerous conditions exist to verify the validity of the growth condition given in Proposition \ref{Dergrowth}; see, e.g., \cite[Lemma 2.5]{Casasbang} or \cite[Theorem 6.3]{Alellip}. A comprehensive discussion on the assumptions pertinent to this growth condition is available in \cite[Section 6]{Alellip}. An explicit example where the growth holds for $\mu\in\mathbb N$ was given in \cite[Example 1.2]{Seyden1} (an optimal control problem  constrained by ODEs).}

 \vskip1cm
\centerline{Alberto Dom{\'i}nguez Corella}
 \centerline{Friedrich-Alexander-Universit\"{a}t Erlangen-N\"{u}rnberg}
\centerline{Department of Data Science, Chair for Dynamics, Control and Numerics (Alexander von
Humboldt-Professorship)}
\centerline{Erlangen, Germany}
 \centerline{E-mail: alberto.of.sonora@gmail.com}
  \vskip0.5cm
\centerline{Nicolai Jork}
 \centerline{Institute of Statistics and Mathematical Methods in Economics}
\centerline{Vienna University of Technology, Austria}
 \centerline{E-mail: nicolai.jork@tuwien.ac.at}
 \vskip0.5cm
\centerline{\v S{\'a}rka Ne\v casov\`a}
 \centerline{Institute of Mathematics of the Academy of Sciences of the Czech Republic}
\centerline{\v Zitna 25, 115 67 Praha 1, Czech Republic}
 \centerline{E-mail: matus@math.cas.cz}
 \vskip0.5cm
\centerline{John Sebastian H. Simon}
\centerline{Johann Radon Institute for Computational and Applied Mathematics (RICAM)}
\centerline{Austrian Academy of Sciences}
\centerline{Altenberger Strasse 69, 4040 Linz, Austria}
\centerline{E-mail: john.simon@ricam.oeaw.ac.at; jhsimon1729@gmail.com}

\end{document}